\documentclass{amsart}
\usepackage{CJKutf8}

\usepackage{mathtools}
\usepackage{etoolbox}
\usepackage{ytableau} 
\makeatletter
\patchcmd{\@settitle}{\uppercasenonmath\@title}{}{}{}
\patchcmd{\@setauthors}{\MakeUppercase}{}{}{}
\patchcmd{\section}{\scshape}{}{}{}
\makeatother
\makeatletter
\patchcmd{\@maketitle}
  {\ifx\@empty\@dedicatory}
  {\ifx\@empty\@date \else {\vskip2ex 
  \centering\footnotesize\@date\par\vskip1ex}\fi
   \ifx\@empty\@dedicatory}
  {}{}
\patchcmd{\@adminfootnotes}
  {\ifx\@empty\@date\else \@footnotetext{\@setdate}\fi}
  {}{}{}
\makeatother

\usepackage{enumitem}
\usepackage{tikz}
\usetikzlibrary{backgrounds}
\usetikzlibrary{matrix}
\usetikzlibrary{arrows}
\usetikzlibrary{shapes,shapes.geometric,shapes.misc}
\tikzstyle{tikzfig}=[baseline=-0.25em,scale=0.75]
\pgfkeys{/tikz/tikzit fill/.initial=0}
\pgfkeys{/tikz/tikzit draw/.initial=0}
\pgfkeys{/tikz/tikzit shape/.initial=0}
\pgfkeys{/tikz/tikzit category/.initial=0}
\pgfdeclarelayer{edgelayer}
\pgfdeclarelayer{nodelayer}
\pgfsetlayers{background,edgelayer,nodelayer,main}

\tikzstyle{none}=[inner sep=0mm]
\tikzstyle{Vert}=[fill=black, draw=black, shape=circle, minimum size=0.4em, inner sep=0.4pt, scale=0.6]
\tikzstyle{label-red}=[fill=white, draw=black, shape=circle, scale=0.4, color=red]
\tikzstyle{label small}=[fill=none, draw=none, shape=circle, scale=0.7, inner sep = 0.1]
\tikzstyle{blt}=[fill=black, draw=black, shape=circle, scale=0.5]
\tikzstyle{G-vert}=[fill=white, draw=black, shape=circle, inner sep=0.7pt, minimum size=0.4em, scale=0.6]
\tikzstyle{label}=[inner sep=0.15mm, font={\footnotesize}]
\tikzstyle{vert}=[fill=black, draw=white, shape=circle, line width=0.25mm, tikzit shape=circle, inner sep=0.35mm]
\tikzstyle{nodes}=[fill=white, draw=none, shape=circle, inner sep=0.7pt, font={\scriptsize}, minimum size=11pt]
\tikzstyle{label-s}=[inner sep=0.1mm, font={\scriptsize}]

\tikzstyle{label-w}=[fill=white, draw=white, shape=circle, inner sep=0.07mm, font={\scriptsize}]
\tikzstyle{particle}=[fill=black, draw=black, shape=circle, inner sep=0.76 mm]

\tikzstyle{dashed}=[-, densely dotted]
\tikzstyle{virtual}=[-, double]
\tikzstyle{RED}=[-, draw=red]
\tikzstyle{CYAN}=[-, draw=cyan]
\tikzstyle{BLUE}=[-, draw=blue]
\tikzstyle{LGREEN}=[-, draw=green]
\tikzstyle{DGREEN}=[-, draw={rgb,255: red,0; green,128; blue,128}]
\tikzstyle{PURPLE}=[-, draw={rgb,255: red,128; green,0; blue,128}]
\tikzstyle{ORANGE}=[-, draw={rgb,255: red,255; green,128; blue,0}]
\tikzstyle{MAGENTA}=[-, draw=magenta]
\tikzstyle{BLUE(dashed)}=[-, draw=blue, densely dotted]
\tikzstyle{GREEN}=[-, draw={rgb,255: red,0; green,208; blue,6}]
\tikzstyle{MAGENTA(dashed)}=[-, draw=magenta, densely dotted]
\tikzstyle{ORANGE(dashed)}=[-, draw={rgb,255: red,255; green,128; blue,0}, densely dotted]
\tikzstyle{RED(dashed)}=[-, draw={rgb,255: red,191; green,0; blue,64}, densely dotted]
\tikzstyle{PURPLE(dashed)}=[-, draw={rgb,255: red,128; green,0; blue,128}, densely dotted]
\tikzstyle{GREEN(dashed)}=[-, draw={rgb,255: red,0; green,208; blue,6}, densely dotted]
\tikzstyle{shade}=[-, opacity=0.4, draw={rgb,255: red,128; green,128; blue,128}, line width=6.5, fill=none, line cap=round,rounded corners]
\tikzstyle{hyper}=[-, fill={rgb,255: red,48; green,255; blue,214}, fill opacity=0.6, draw={rgb,255: red,18; green,229; blue,85}, tikzit fill={rgb,255: red,48; green,255; blue,214}]
\tikzstyle{hyper1}=[-, fill={rgb,255: red,230; green,138; blue,9}, draw={rgb,255: red,255; green,128; blue,0}, fill opacity=0.5]
\tikzstyle{hyper2}=[-, fill={rgb,255: red,128; green,179; blue,255}, draw={rgb,255: red,46; green,87; blue,115}, fill opacity=0.55]
\tikzstyle{new edge style 0}=[-, fill={rgb,255: red,245; green,255; blue,39}, draw={rgb,255: red,168; green,170; blue,22}, fill opacity=0.6]
\tikzstyle{blue}=[-, draw={rgb,255: red,49; green,149; blue,255}, line width=1.5pt]
\tikzstyle{red}=[-, draw=red, line width=1.5pt]
\tikzstyle{blue thick}=[-, tikzit draw=blue, draw=blue, line width=1.7pt]
\tikzstyle{thick}=[-, ultra thick]
\tikzstyle{dotted}=[-, densely dotted]
\tikzstyle{e1}=[-, color=1, line width=1]
\tikzstyle{e2}=[-, color=2, line width=1]
\tikzstyle{e3}=[-, color=3, line width=1]
\tikzstyle{e4}=[-, color=4, line width=1]
\tikzstyle{e5}=[-, color=5, line width=1]
\tikzstyle{e6}=[-, color=6, line width=1]
\tikzstyle{e7}=[-, color=7, line width=1]
\tikzstyle{e8}=[-, color=8, line width=1]
\tikzstyle{e9}=[-, color=9, line width=1]
\tikzstyle{e0}=[-, color=0, line width=1]
\tikzstyle{ev}=[-, color=v, line width=1]
\tikzstyle{d1}=[-, color=1, line width=0.8, densely dotted]
\tikzstyle{d2}=[-, color=2, line width=0.8, densely dotted]
\tikzstyle{d3}=[-, color=3, line width=0.8, densely dotted]
\tikzstyle{d4}=[-, color=4, line width=0.8, densely dotted]
\tikzstyle{d5}=[-, color=5, line width=0.8, densely dotted]
\tikzstyle{d6}=[-, color=6, line width=0.8, densely dotted]
\tikzstyle{d7}=[-, color=7, line width=0.8, densely dotted]
\tikzstyle{d8}=[-, color=8, line width=0.8, densely dotted]
\tikzstyle{d9}=[-, color=9, line width=0.8, densely dotted]
\tikzstyle{d0}=[-, color=0, line width=0.8, densely dotted]

\tikzstyle{m1}=[-, line width=2]
\tikzstyle{red-fill}=[-, fill=red, fill opacity=0.5,draw = none]
\tikzstyle{blue-fill}=[-, fill=blue, fill opacity=0.5, draw = none]

\tikzset{%
every path/.append style={line width = 0.8 pt}
}

\definecolor{2}{rgb}{0.91, 0.33, 0.5}
\definecolor{1}{rgb}{0.0, 1.0, 1.0} 
\definecolor{8}{rgb}{1.0, 0.49, 0.0}
\definecolor{3}{rgb}{0.71, 0.49, 0.86}
\definecolor{6}{rgb}{0.5, 0, 0}
\definecolor{5}{rgb}{0.0, 0.5, 1.0}
\definecolor{0}{rgb}{0.6,0.6,0.6}
\definecolor{7}{rgb}{1.0, 0.75, 0.0}
\definecolor{4}{rgb}{0.2, 0.2, 0.6}
\definecolor{9}{rgb}{0.24, 0.57, 0.22}
\definecolor{v}{rgb}{0.54, 0.47, 1.12}

\usepackage{stmaryrd}
\usepackage{mathtools}
\def\multichoose#1#2{\ensuremath{\left(\kern-.3em\left(\genfrac{}{}{0pt}{}{#1}{#2}\right)\kern-.3em\right)}}
\def\mbinom#1#2{\ensuremath{\left(\kern-.3em\left(\genfrac{}{}{0pt}{}{#1}{#2}\right)\kern-.3em\right)}}

\usepackage{mathabx}

\usepackage{accents}

\newcommand{\tup}[1]{\langle{#1}\rangle}

\newcommand{\north}{\mathrm{north}}

\renewcommand{\O}{\mathcal{O}}
\newcommand{\ord}{\mathrm{ord}}

\DeclareMathOperator{\cc}{\mathbb{C}}

\DeclareMathOperator{\oto}{\Omega_{\textup{north}}}
\DeclareMathOperator{\ori}{\Omega_{\textup{east}}}
\DeclareMathOperator{\evac}{evac}

\makeatother

\usepackage{graphicx}
\usepackage{color}
\usepackage{etoolbox}
\usepackage[margin=1in]{geometry}
\usepackage{amsmath,amsthm,amssymb,tikz,tikz-cd}
\usepackage[backref=page]{hyperref}
\hypersetup{
    colorlinks = true,
    citecolor = orange,%
}
\usepackage[noabbrev]{cleveref}
\usepackage{subcaption}

\usepackage{tikz-cd}

\theoremstyle{definition}
\newtheorem{theorem}{Theorem}[section]

\newtheorem{lemma}[theorem]{Lemma}
\newtheorem{prop}[theorem]{Proposition}
\newtheorem{conj}[theorem]{Conjecture}

\newtheorem{corollary}[theorem]{Corollary}

\theoremstyle{definition}
\newtheorem{remark}[theorem]{Remark}
\newtheorem{example}[theorem]{Example}

\newtheorem{definition}[theorem]{Definition}

\DeclareMathOperator{\zz}{\mathbb{Z}}

\renewcommand{\L}{\mathrm{L}}
\renewcommand{\S}{\mathrm{s}}
\newcommand{\C}{\mathrm{R}}
\newcommand{\K}{\mathcal{K}}

\DeclareMathOperator{\col}{Col}
\DeclareMathOperator{\row}{Row}

\DeclareMathOperator{\sh}{sh}

\newcommand{\RS}{\operatorname{RS}}

\title[]{Dual Affine Robinson-Schensted Correspondence}

\author[Daoji Huang and Sylvester Zhang]{%
\parbox{3cm}{\centering
Daoji Huang$^\flat$\\
\begin{CJK}{UTF8}{gkai}
黄道骥
\end{CJK}
}%
\hspace{1em}%
\parbox{3cm}{\centering
Sylvester W. Zhang$^\sharp$\\
\begin{CJK}{UTF8}{gkai}
张文泽
\end{CJK}
}
}

\thanks{$^\flat$\href{mailto:daojihuang@umass.edu}{daojihuang@umass.edu} University of  Massachusetts Amherst}

\thanks{$^\sharp$\href{mailto:sylvesterzhang@math.ucla.edu}{sylvesterzhang@math.ucla.edu} University of California Los Angeles}
\date{}

\begin{document}
\setcounter{tocdepth}{1}
\maketitle
\begin{abstract}
	We introduce the dual affine Robinson–Schensted correspondence that gives a bijection between the extended affine symmetric group and tuples $(\bar{P},\bar{Q},\lambda,N)$, where $\bar{P}$ and $\bar{Q}$ are tabloids, $\lambda$ is a partition, and $N$ is an integer, subject to compatibility conditions. The construction generalizes Fomin's growth diagrams and Viennot's shadow lines for the classical Robinson-Schensted correspondence on the symmetric group, and is dual to the affine matrix ball construction as well as Shi's correspondence, in the sense that the $P$-tabloids are the same, and the $Q$-tabloids are related by affine evacuation. As a consequence, our construction also parametrizes Kazhdan–Lusztig cells in affine type $A$. We conjecture that the growth diagrams we construct admit a natural geometric realization in terms of relative positions of affine flags, similar to the interpretation given by Steinberg and van Leeuwen in the classical case. 
	\end{abstract}
\tableofcontents


\section{Introduction}

The classical Robinson--Schensted correspondence  is a bijection between the symmetric group $S_n$ and the set of pairs of standard Young tableaux of the same shape of size $n$. It
 plays a fundamental role in representation theory and geometry through two applications. On the one hand, it parametrizes Kazhdan--Lusztig cells in the symmetric group. On the other hand, it admits a geometric interpretation in terms of the relative position of pairs of flags in $G/B$ \cite{steinberg_rs}. Together, these applications show that the Robinson--Schensted correspondence provides a natural bridge between the theory of Kazhdan-Lusztig cells and the geometry of flag varieties. Combinatorially, there are several different well-known constructions of the Robinson--Schensted correspondence, including Schensted's insertion algorithm and diagrammatic constructions of Fomin, Viennot, and Fulton. Each approach has its own unique advantage revealing different structures of the bijection: crudely speaking, the insertion algorithm is most amenable for computation and experimentation, while the diagrammatic approaches tend to reveal deeper properties, such as symmetries, of the correspondence. Notably, Fomin's growth diagrams are particularly well-suited for geometry: they can be directly interpreted as arising from the Jordan types of the restrictions of a generic nilpotent operator which simultaneously contracts a pair of complete flags in relative position $w$  to all pairwise intersections of the subspaces in the two flags, thereby making the link between combinatorics and geometry especially explicit.
 We refer the reader to the excellent exposition \cite{fominbritz} that explains this perspective, as well as further connections to the Green-Kleitman duality in poset theory.
  
 Motivated by the constructions of Steinberg, Spaltenstein, and Springer, Lusztig \cite{lusztig} introduced $S$-cells as the image of the relative position map of irreducible components of a Springer fiber, and noted that this partition of the Weyl group is somewhat similar to the partition by two-sided cells in terms of Hecke algebras. Lusztig noted that they coincide in type $A$ but diverge for other classical types. He also asked for further investigation of $S$-cells and two-sided cells in affine types, since the theories of Hecke algebras and Springer fibers generalize. 
 
 In affine type~$A$, Shi introduced an affine Robinson--Schensted correspondence that associates to each affine permutation a pair of tabloids $\bar P,\bar Q$. The fundamental feature of Shi's construction is that two affine permutations belong to the same right cell if and only they have the same $\bar P$-tabloid, thereby providing an affine analogue of the classical cell parametrization \cite{shi}. Shi's construction is a sophisticated insertion algorithm on tabloids. This, however, is far from a bijection, as it maps an infinite set to a finite set. Honeywill \cite{honeywill} refined Shi's map to a bijection by adding a third piece of data which is a dominant weight to the image.
 Chmutov, Pylyavskyy, and Yudovina \cite{ambc} introduced the affine matrix ball construction, which provides a diagrammatic interpretation of the affine Robinson--Schensted correspondence, and gave an asymptotic interpretation of Shi's correspondence using the usual Robinson--Schensted, which was studied earlier by Pak \cite{pak}. In the manuscript \cite{pablo},
 Boixeda, Ying, and Yue applied the affine matrix ball construction and proved that the $S$-cells and two-sided cells coincide in affine type $A$ in the ``rectangular type'' case; the general case remains open.
 
 In this paper, we introduce a \emph{dual affine Robinson--Schensted correspondence}%
\[
w \longmapsto \bigl(\bar{P}, \bar{Q}, \lambda, N\bigr)
\]%
for the extended affine symmetric group, where $\bar{P}, \bar{Q}$ are tabloids of the same shape, $\lambda$ is a partition, and $N\in\mathbb{Z}$ satisfying certain compatibility conditions. Our construction simultaneously generalizes Fomin's growth diagrams and Viennot's shadow line construction, and it is dual to Shi's correspondence and the affine matrix ball construction in the sense that the resulting $P$-tableaux are the same and the $Q$-tableaux differ by the affine evacuation map in the sense of \cite{chmutov2022affine}. This resembles what happens in the usual setting: if inserting a permutation $w$ read from left to right via row insertion gives $P$ and $Q$ and inserting $w$ read from right to left via column insertion gives $P'$ and $Q'$, then $P=P'$ and $Q$ and $Q'$ are related by the evacuation map. As a result, our construction recovers the known affine Kazhdan-Lusztig cell structure. A unique feature of our map is a coloring rule on the shadow lines which is intrinsic to the affine setting that reveals the meaning of row indices in the asymptotic insertion. As compared to the affine matrix ball construction, it appears that our map and its combinatorial analysis in establishing the bijection is much simpler.
 
 We conjecture that our growth diagram construction has a natural geometric interpretation in terms of Jordan types of restrictions of certain topologically nilpotent operators. Roughly speaking, given two affine flags $E_\bullet$ and $F_\bullet$ in $GL_n(\mathbb{C}((t^{-1})))/I$ (where $I$ is the Iwahori subgroup) interpreted as $\mathbb{Z}$-indexed chains of $\mathbb{C}[[t^{-1}]]$-lattices in relative position $w$, and $\mu$ a \emph{generic} regular semisimple, topologically nilpotent element satisfying $\mu E_i\subset E_{i-1}$ and $\mu F_i\subset F_{i-1}$, we conjecture that the partitions in our growth diagram can be interpreted as the Jordan types of the operator $\mu$ restricted to the finite dimensional spaces $E_i/(E_i\cap F_j)$ for all $i,j\in \mathbb{Z}$. (In fact, it was this conjecture that led us to the discovery of the combinatorial map.) This can be seen as a geometric motivation of the dual perspective: in the affine setting, the intersections $E_i\cap F_j$, which would correspond to the interpretation of the usual Robinson-Schensted, are infinite dimensional $\mathbb{C}$-vector spaces and thus do not have a readily available Jordan theory, whereas the quotients $E_i/(E_i\cap F_j)$ are finite dimensional. This distinction is not essential in the finite case but pronounced in the affine case. In this light, we hope our combinatorial construction provides new tools that are potentially geometrically natural for further  investigation in Kazhdan--Lusztig theory.
 
This paper is structured as follows.  
\Cref{sec:growth} reviews the classical constructions of Fomin and Viennot on growth diagrams for Robinson-Schensted and their relationship that inspires our construction, as well as the dual version. \Cref{sec:affine} gives precise statements of our main theorem (\Cref{thm:main}) and examples and \Cref{sec:proof} gives the proof. \Cref{sec:KL} connects our construction to existing theory of Kazhdan Lusztig cells in affine type $A$ and justifies the adjective``dual'' via affine evacuation (\Cref{thm:dual}). Finally, \Cref{sec:geometry} gives a conjectural geometric interpretation of our construction ({\Cref{conj:main-geometry}).

\noindent\textbf{Acknowledgements.} We thank Pavlo Pylyavskyy, Jake Levinson, Joel Kamnitzer, Pablo Boixeda Alvarez, and Do Kien Hoang for helpful discussions. The figures in this paper were generated by ChatGPT. DH was partially supported by NSF-DMS2202900.

\section{Fomin-Viennot Growth Diagram}
\label{sec:growth}
In this mainly expository section, we recall the classical Robinson-Schensted correspondence and set our conventions, using Fomin's \emph{growth diagram} and Viennot's \emph{shadow line construction}. Throughout the paper, we use matrix coordinates.

\subsection{Growth Diagrams} Given a permutation $w\in S_n$, we construct its \emph{growth diagram} $G_w$ as follows. First, start with an $n\times n$ grid, in which the 
vertices have coordinates $(i,j)$ for $0\le i,j\le n$. 
We call the tile with the four vertices $(i-1,j-1)$, $(i-1,j)$, $(i,j-1)$, $(i,j)$ the $(i,j)$-tile.
The
$(i,j)$-tile is marked with a dot if $w(j)=i$, where $1\le i,j\le n$. This is sometimes called a \emph{rock diagram}.
For example, the rock diagram of $365214$ is

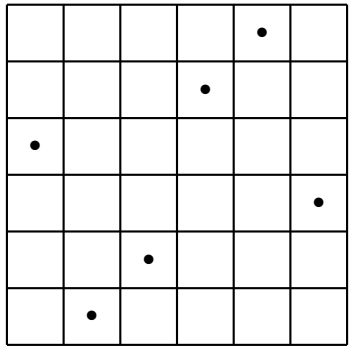
\begin{center}

	\begin{tikzpicture}[scale = 0.75]
		\foreach \i in {0,...,6} {
		\draw (0,\i ) -- (6,\i);
		}
		\foreach \j in {0,...,6}{
		\draw (\j,0) -- (\j,6);
		}
	
		\node () at (1-0.5,4-0.5) {$\bullet$};
		\node () at (2-0.5,1-0.5) {$\bullet$};
		\node () at (3-0.5,2-0.5) {$\bullet$};
		\node () at (4-0.5,5-0.5) {$\bullet$};
		\node () at (5-0.5,6-0.5) {$\bullet$};
		\node () at (6-0.5,3-0.5) {$\bullet$};
\end{tikzpicture}
\end{center}

A growth diagram is, loosely speaking, a way to assign partitions to each vertex of the rock diagram, following a set of local rules.

First, place $\varnothing$ at each vertex of the north and west boundaries. Then in any box, given the northwest, northeast, and southwest partitions, the southeast partition is uniquely determined by the following rules, which generate the entire growth diagram $G_w$.

\begin{definition}[Vertex local rules]\label{def:local_rules}
For an unmarked square \begin{tikzpicture}[scale = 1.2,baseline = -7mm,yscale = -1]
	\draw (0,0) -- (1,0) -- (1,1) -- (0,1) -- cycle;
	\node [fill=white]() at (0,0) {$\lambda_{i,j}$};
	\node [fill=white]() at (1,0) {$\lambda_{i+1,j}$};
	\node [fill=white]() at (1.2,1) {$\lambda_{i+1,j+1}$};
	\node [fill=white]() at (-0.2,1) {$\lambda_{i,j+1}$};
\end{tikzpicture}, we have
\begin{enumerate}
\item If $\lambda_{i,j+1}\neq \lambda_{i+1,j}$, then $\lambda_{i+1,j+1}=\lambda_{i,j+1}\cup \lambda_{i+1,j}$, i.e. \begin{tikzpicture}[scale = 1.2,baseline = -7mm,yscale = -1]
	\draw (0,0) -- (1,0) -- (1,1) -- (0,1) -- cycle;
	\node [fill=white]() at (0,0) {$\lambda$};
	\node [fill=white]() at (1,0) {$\mu_1$};
	\node [fill=white]() at (1.2,1) {$\mu_1\cup \mu_2$};
	\node [fill=white]() at (0,1) {$\mu_2$};
\end{tikzpicture}
As a special case, if $\lambda_{i,j}=\lambda_{i+1,j}$, then $\lambda_{i+1,j+1}=\lambda_{i,j+1}$, i.e. $\begin{tikzpicture}[scale = 1,baseline = -7mm,yscale = -1]
	\draw (0,0) -- (1,0) -- (1,1) -- (0,1) -- cycle;
	\node [fill=white]() at (0,0) {$\lambda$};
	\node [fill=white]() at (1,0) {$\lambda$};
	\node [fill=white]() at (1,1) {$\mu$};
	\node [fill=white]() at (0,1) {$\mu$};
\end{tikzpicture}$

\item If $\lambda_{i,j+1} = \lambda_{i+1,j}\neq \lambda_{i,j}$, then \begin{tikzpicture}[scale = 1,baseline = -7mm,yscale = -1]
	\draw (0,0) -- (1,0) -- (1,1) -- (0,1) -- cycle;
	\node [fill=white]() at (0,0) {$\lambda$};
	\node [fill=white]() at (1,0) {$\mu$};
	\node [fill=white]() at (1,1) {$\lambda'$};
	\node [fill=white]() at (0,1) {$\mu$};
\end{tikzpicture}, where $\lambda'$ is obtained by adding one box to $\lambda$  to the row immediately below $\mu/\lambda$\footnote{In this case, $\mu$ will always have exactly one more box than $\lambda$.}.

\end{enumerate}
For a marked square, we have
\begin{enumerate}
	\item [(3)] \begin{tikzpicture}[scale = 1,baseline = -7mm, yscale = -1]
	\node at (0.5,0.5) {$\bullet$};
	\draw (0,0) -- (1,0) -- (1,1) -- (0,1) -- cycle;
	\node [fill=white]() at (0,0) {$\lambda$};
	\node [fill=white]() at (1,0) {$\lambda$};
	\node [fill=white]() at (1,1) {$\lambda'$};
	\node [fill=white]() at (0,1) {$\lambda$};
\end{tikzpicture}, where $\lambda'$ is obtained by adding one box to the first row of $\lambda$.\end{enumerate}
\end{definition}
Then the bottom (resp. right) boundary of $G_w$ is a nested sequence of partitions, which is equivalent to a standard Young tableau\footnote{Recall that a standard Young tableau (SYT) is a filling of a Young diagram for a partition of $n$ by numbers in $[n]$ such that the numbers are increasing from left to right in each row, and increasing from top to bottom in each column.}, denoted $P(w)$ (resp. $Q(w)$). The map $\RS:w\mapsto (P(w),Q(w))$ is known as the Robinson-Schensted correspondence. 

Geometrically, for $w\in S_n$, if $E_\bullet$ is the standard flag, namely, $E_i=\mathrm{span}_\mathbb{C}\{e_1,\dots, e_i\}$ for $0\le i\le n$ and $F_\bullet$ is the $w$-coordinate flag, namely $F_i=\mathrm{span}_\mathbb{C}\{e_{w(1)},\dots, e_{w(i)}\}$ for $0\le i\le n$, and $\mu$ is a generic nilpotent matrix satisfying $\mu E_i\subseteq E_{i-1}$ and $\mu F_i\subseteq F_{i-1}$ for all $0\le i\le n$, then $(G_w)_{i,j}$ is the Jordan type of $\mu$ restricted to $E_i\cap F_j$. This is explained in \cite[Section 3 and 6]{fominbritz}. A related interpretation is given in \cite{vanLeeuwen_flag}.

\subsection{Shadow Line Construction}

On the other hand, Viennot gave a {``geometric'' construction} of the Robinson-Schensted correspondence, which we call the \emph{shadow line construction}. This is also known as the \emph{matrix ball construction} \cite{fulton}. Here we briefly describe the construction.

Again, we start from the rock diagram of a permutation, and imagine that there is a light that shines from the northwest corner to the southeast. In the first iteration, we start by considering the marks $(j,w(j))$ as poles and cast shadow to the south and east, which will result in a shadow line that is the boundary of the union of the ``rectangular shadows''. After obtaining the first shadow line, we disregard it and repeat the process on the remaining marks until all marks lie on a shadow line. We  then start the second iteration and plant new poles on the corners of the shadow lines constructed in the first iteration that do not have a pole, remove the old poles from the previous iteration, and repeat the process as described in the first iteration. We iterate until no new poles can be planted.
 \Cref{fig:finite_shadow_line} is an example of a shadow line construction. 

\begin{figure}
	\begin{tikzpicture}
	\begin{pgfonlayer}{nodelayer}
		\node [style=none] (0) at (-3, 3) {};
		\node [style=none] (1) at (3, 3) {};
		\node [style=none] (2) at (-3, 2) {};
		\node [style=none] (3) at (3, 2) {};
		\node [style=none] (4) at (-3, 1) {};
		\node [style=none] (5) at (3, 1) {};
		\node [style=none] (6) at (-3, 0) {};
		\node [style=none] (7) at (3, 0) {};
		\node [style=none] (8) at (-3, -1) {};
		\node [style=none] (9) at (3, -1) {};
		\node [style=none] (10) at (-3, -2) {};
		\node [style=none] (11) at (3, -2) {};
		\node [style=none] (12) at (-3, -3) {};
		\node [style=none] (13) at (3, -3) {};
		\node [style=none] (16) at (-3, 3) {};
		\node [style=none] (17) at (-3, -3) {};
		\node [style=none] (18) at (-2, 3) {};
		\node [style=none] (19) at (-2, -3) {};
		\node [style=none] (20) at (-1, 3) {};
		\node [style=none] (21) at (-1, -3) {};
		\node [style=none] (22) at (0, 3) {};
		\node [style=none] (23) at (0, -3) {};
		\node [style=none] (24) at (1, 3) {};
		\node [style=none] (25) at (1, -3) {};
		\node [style=none] (26) at (2, 3) {};
		\node [style=none] (27) at (2, -3) {};
		\node [style=none] (28) at (3, 3) {};
		\node [style=none] (29) at (3, -3) {};
		\node [style=blt] (30) at (-2.5, 0.5) {};
		\node [style=blt] (31) at (-1.5, -2.5) {};
		\node [style=blt] (32) at (-0.5, -1.5) {};
		\node [style=blt] (33) at (0.5, 1.5) {};
		\node [style=blt] (34) at (1.5, 2.5) {};
		\node [style=blt] (35) at (2.5, -0.5) {};
		\node [style=none] (36) at (-2.5, -3) {};
		\node [style=none] (37) at (-1.5, -3) {};
		\node [style=none] (38) at (-0.5, -3) {};
		\node [style=none] (39) at (0.5, -3) {};
		\node [style=none] (40) at (1.5, -3) {};
		\node [style=none] (41) at (2.5, -3) {};
		\node [style=none] (42) at (3, -2.5) {};
		\node [style=none] (43) at (3, -1.5) {};
		\node [style=none] (44) at (3, -0.5) {};
		\node [style=none] (45) at (3, 0.5) {};
		\node [style=none] (46) at (3, 1.5) {};
		\node [style=none] (47) at (3, 2.5) {};
		\node [style=none] (48) at (0.5, 0.5) {};
		\node [style=none] (49) at (1.5, 1.5) {};
		\node [style=none] (50) at (-0.5, -2.5) {};
		\node [style=none] (51) at (2.5, -1.5) {};
		\node [style=none] (52) at (-2.5, 0.5) {};
		\node [style=none] (53) at (0.5, -2.5) {};
		\node [style=none] (54) at (1.5, 0.5) {};
	\end{pgfonlayer}
	\begin{pgfonlayer}{edgelayer}
		\draw (0.center) to (1.center);
		\draw (2.center) to (3.center);
		\draw (4.center) to (5.center);
		\draw (6.center) to (7.center);
		\draw (8.center) to (9.center);
		\draw (10.center) to (11.center);
		\draw (12.center) to (13.center);
		\draw (16.center) to (17.center);
		\draw (18.center) to (19.center);
		\draw (20.center) to (21.center);
		\draw (22.center) to (23.center);
		\draw (24.center) to (25.center);
		\draw (26.center) to (27.center);
		\draw (28.center) to (29.center);
		\draw [style=e1] (36.center) to (52.center);
		\draw [style=e1] (52.center) to (48.center);
		\draw [style=e1] (48.center) to (33);
		\draw [style=e1] (33) to (49.center);
		\draw [style=e1] (49.center) to (34);
		\draw [style=e1] (34) to (47.center);
		\draw [style=e1] (37.center) to (31);
		\draw [style=e1] (31) to (50.center);
		\draw [style=e1] (50.center) to (32);
		\draw [style=e1] (32) to (51.center);
		\draw [style=e1] (51.center) to (35);
		\draw [style=e1] (35) to (44.center);
		\draw [style=e7] (38.center) to (50.center);
		\draw [style=e7] (50.center) to (53.center);
		\draw [style=e7] (53.center) to (48.center);
		\draw [style=e7] (48.center) to (54.center);
		\draw [style=e7] (54.center) to (49.center);
		\draw [style=e7] (49.center) to (46.center);
		\draw [style=e7] (41.center) to (51.center);
		\draw [style=e7] (51.center) to (43.center);
		\draw [style=e3] (40.center) to (54.center);
		\draw [style=e3] (54.center) to (45.center);
	\end{pgfonlayer}
\end{tikzpicture}
	\caption{Example of the shadow line construction}
	\label{fig:finite_shadow_line}
\end{figure}
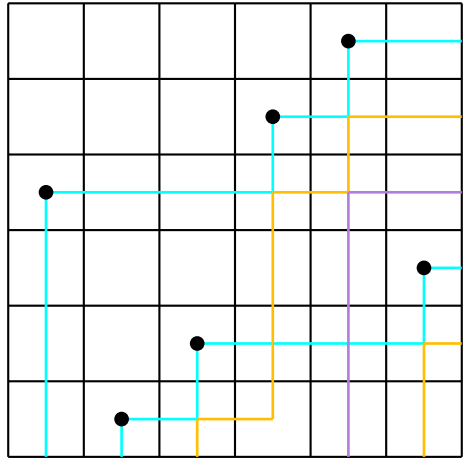

\subsection{Edge Local Rules}The above-mentioned two approaches of Robinson--Schensted correspondence can be put under one roof using  edge local rules.

If we label the edges of the growth diagram by $i$ if the two partitions on that edge differ by a box in row $i$, the local rules in \Cref{def:local_rules} ensure that only the edge labels as in \Cref
{fig:edge_rule} will show up. These are called the \emph{edge local rules} by X. Viennot. 

\begin{figure}
	\[
\begin{tikzpicture}[tikzfig]
	\begin{pgfonlayer}{nodelayer}
		\node [style=none] (0) at (-9, 3) {};
		\node [style=label-s] (1) at (-8, 3) {$0$};
		\node [style=none] (2) at (-7, 3) {};
		\node [style=label-s] (3) at (-7, 2) {$1$};
		\node [style=none] (4) at (-7, 1) {};
		\node [style=label-s] (5) at (-8, 1) {$1$};
		\node [style=none] (6) at (-9, 1) {};
		\node [style=label-s] (7) at (-9, 2) {$0$};
		\node [style=none] (8) at (-6, 3) {};
		\node [style=label-s] (9) at (-5, 3) {$0$};
		\node [style=none] (10) at (-4, 3) {};
		\node [style=label-s] (11) at (-4, 2) {$0$};
		\node [style=none] (12) at (-4, 1) {};
		\node [style=label-s] (13) at (-5, 1) {$0$};
		\node [style=none] (14) at (-6, 1) {};
		\node [style=label-s] (15) at (-6, 2) {$0$};
		\node [style=none] (16) at (-3, 3) {};
		\node [style=label-s] (17) at (-2, 3) {$a$};
		\node [style=none] (18) at (-1, 3) {};
		\node [style=label-s] (19) at (-1, 2) {$a+1$};
		\node [style=none] (20) at (-1, 1) {};
		\node [style=label-s] (21) at (-2, 1) {$a+1$};
		\node [style=none] (22) at (-3, 1) {};
		\node [style=label-s] (23) at (-3, 2) {$a$};
		\node [style=none] (24) at (0, 3) {};
		\node [style=label-s] (25) at (1, 3) {$a$};
		\node [style=none] (26) at (2, 3) {};
		\node [style=label-s] (27) at (2, 2) {$b$};
		\node [style=none] (28) at (2, 1) {};
		\node [style=label-s] (29) at (1, 1) {$a$};
		\node [style=none] (30) at (0, 1) {};
		\node [style=label-s] (31) at (0, 2) {$b$};
		\node [style=particle] (32) at (-8, 2) {};
		\node [style=label] (33) at (-2, 0.25) {$a\neq 0$};
		\node [style=label] (34) at (1, 0.25) {$a\neq b$};
	\end{pgfonlayer}
	\begin{pgfonlayer}{edgelayer}
		\draw (0.center) to (1);
		\draw (1) to (2.center);
		\draw (2.center) to (3);
		\draw (3) to (4.center);
		\draw (4.center) to (5);
		\draw (5) to (6.center);
		\draw (6.center) to (7);
		\draw (7) to (0.center);
		\draw (8.center) to (9);
		\draw (9) to (10.center);
		\draw (10.center) to (11);
		\draw (11) to (12.center);
		\draw (12.center) to (13);
		\draw (13) to (14.center);
		\draw (14.center) to (15);
		\draw (15) to (8.center);
		\draw (16.center) to (17);
		\draw (17) to (18.center);
		\draw (18.center) to (19);
		\draw (19) to (20.center);
		\draw (20.center) to (21);
		\draw (21) to (22.center);
		\draw (22.center) to (23);
		\draw (23) to (16.center);
		\draw (24.center) to (25);
		\draw (25) to (26.center);
		\draw (26.center) to (27);
		\draw (27) to (28.center);
		\draw (28.center) to (29);
		\draw (29) to (30.center);
		\draw (30.center) to (31);
		\draw (31) to (24.center);
	\end{pgfonlayer}
\end{tikzpicture}
\]
\caption{Edge local rules.}
\label{fig:edge_rule}
\end{figure}
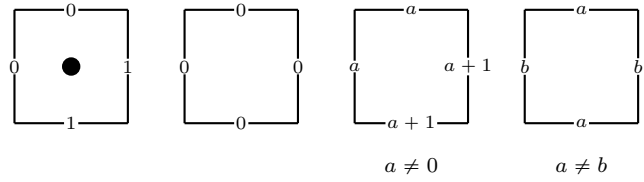

On the other hand, going from the shadow line configuration, we label an edge of the rock diagram by $i$ if there is a shadow line from the $i$-th step that goes through the edge. The edge labels obtained in this way are exactly the same as those determined by the edge local rules \cite{viennot_slides}. This then establishes an equivalence between the two approaches. In \Cref{fig:fomingrowth}, we show the RS correspondence for $w=365214$, $w\mapsto (P(w),Q(w))$,
where $P(w)=\ytableaushort{14, 25,3,6}$ and $Q(w)=\ytableaushort{12,36,4,5}$.

\begin{figure}
\resizebox{0.65\linewidth}{0.65\linewidth}{
	\input{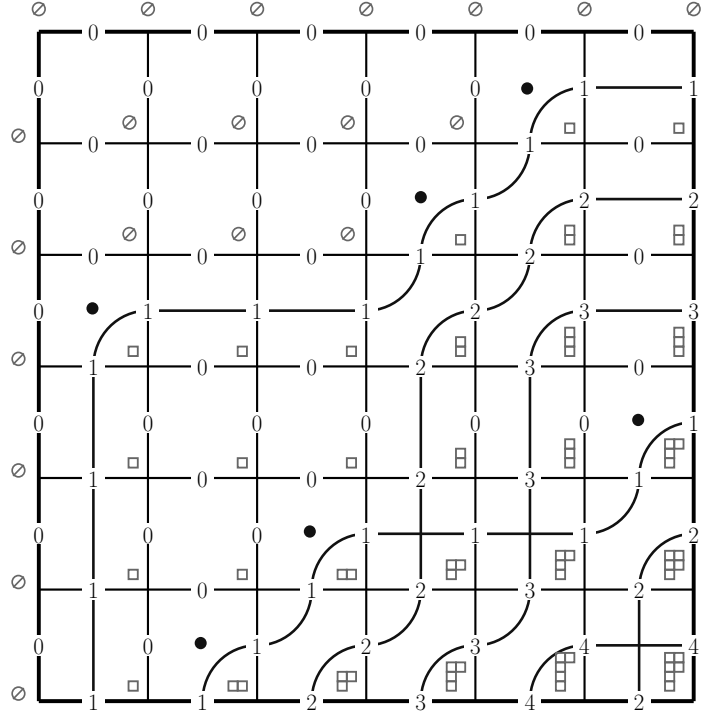}}

	\caption{Robinson-Schensted via vertex and edge local rules}
	\label{fig:fomingrowth}
\end{figure}

The edge local rules have a natural algebraic interpretation using Fomin's Schur operators.
In \cite{fomin}, Fomin introduced a set of operators on Young's lattice, defined as follows.
$$u_i(\lambda)=\begin{cases}
	\lambda \cup\{\text{a box in row }i\}&\text{ if the result is a partition,}\\
	\varnothing &\text{ otherwise.}
\end{cases}$$
$$d_i(\lambda)=\begin{cases}
	\lambda -\{\text{a box in row }i\}&\text{ if the result is a partition,}\\
	\varnothing&\text{ otherwise.}
\end{cases}$$

\begin{prop}[\cite{fomin}]\label{prop:small_u}
	The $u_i$'s and $d_i$'s satisfy the following relations.
	\begin{align*}
	d_1u_1 &= \textup{id}\\
	d_ju_i&=u_{i}d_j\quad i\neq j\\
		d_{i+1}u_{i+1}&=u_id_i.
	\end{align*}
\end{prop}
The edge local rules can be thought of as the commutative diagrams for the relations of these operators, where a vertical (resp. horizontal) edge labeled by $i$ corresponds to the operator $u_i$ (resp. $d_i$). Then the three  edge local rules precisely match the above relations of the Schur operators.

\subsection{Dual RS}
Here we define a ``dual'' RS correspondence, by taking the shadow line from a different direction. This construction will be useful in future sections, and its name is parallel to that of dual affine RS correspondence.

More precisely, we define \emph{dual growth diagram} using dual edge local rules from the same rock diagram. The dual edge local rules are listed as follows (Figure~\ref{fig:dual_edge_rule}).
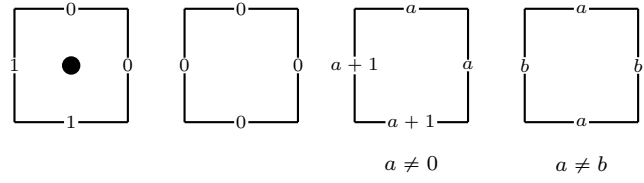
\begin{figure}[h]
	\[
\begin{tikzpicture}[tikzfig]
	\begin{pgfonlayer}{nodelayer}
		\node [style=none] (0) at (-9, 3) {};
		\node [style=label-s] (1) at (-8, 3) {$0$};
		\node [style=none] (2) at (-7, 3) {};
		\node [style=label-s] (3) at (-7, 2) {$0$};
		\node [style=none] (4) at (-7, 1) {};
		\node [style=label-s] (5) at (-8, 1) {$1$};
		\node [style=none] (6) at (-9, 1) {};
		\node [style=label-s] (7) at (-9, 2) {$1$};
		\node [style=none] (8) at (-6, 3) {};
		\node [style=label-s] (9) at (-5, 3) {$0$};
		\node [style=none] (10) at (-4, 3) {};
		\node [style=label-s] (11) at (-4, 2) {$0$};
		\node [style=none] (12) at (-4, 1) {};
		\node [style=label-s] (13) at (-5, 1) {$0$};
		\node [style=none] (14) at (-6, 1) {};
		\node [style=label-s] (15) at (-6, 2) {$0$};
		\node [style=none] (16) at (-3, 3) {};
		\node [style=label-s] (17) at (-2, 3) {$a$};
		\node [style=none] (18) at (-1, 3) {};
		\node [style=label-s] (19) at (-1, 2) {$a$};
		\node [style=none] (20) at (-1, 1) {};
		\node [style=label-s] (21) at (-2, 1) {$a+1$};
		\node [style=none] (22) at (-3, 1) {};
		\node [style=label-s] (23) at (-3, 2) {$a+1$};
		\node [style=none] (24) at (0, 3) {};
		\node [style=label-s] (25) at (1, 3) {$a$};
		\node [style=none] (26) at (2, 3) {};
		\node [style=label-s] (27) at (2, 2) {$b$};
		\node [style=none] (28) at (2, 1) {};
		\node [style=label-s] (29) at (1, 1) {$a$};
		\node [style=none] (30) at (0, 1) {};
		\node [style=label-s] (31) at (0, 2) {$b$};
		\node [style=particle] (32) at (-8, 2) {};
		\node [style=label] (33) at (-2, 0.25) {$a\neq 0$};
		\node [style=label] (34) at (1, 0.25) {$a\neq b$};
	\end{pgfonlayer}
	\begin{pgfonlayer}{edgelayer}
		\draw (0.center) to (1);
		\draw (1) to (2.center);
		\draw (2.center) to (3);
		\draw (3) to (4.center);
		\draw (4.center) to (5);
		\draw (5) to (6.center);
		\draw (6.center) to (7);
		\draw (7) to (0.center);
		\draw (8.center) to (9);
		\draw (9) to (10.center);
		\draw (10.center) to (11);
		\draw (11) to (12.center);
		\draw (12.center) to (13);
		\draw (13) to (14.center);
		\draw (14.center) to (15);
		\draw (15) to (8.center);
		\draw (16.center) to (17);
		\draw (17) to (18.center);
		\draw (18.center) to (19);
		\draw (19) to (20.center);
		\draw (20.center) to (21);
		\draw (21) to (22.center);
		\draw (22.center) to (23);
		\draw (23) to (16.center);
		\draw (24.center) to (25);
		\draw (25) to (26.center);
		\draw (26.center) to (27);
		\draw (27) to (28.center);
		\draw (28.center) to (29);
		\draw (29) to (30.center);
		\draw (30.center) to (31);
		\draw (31) to (24.center);
	\end{pgfonlayer}
\end{tikzpicture}
\]
\caption{Dual edge local rules.}
\label{fig:dual_edge_rule}
	
\end{figure}
To construct a dual growth diagram for $w\in S_n$ using these edge local rules, we start with the rock diagram for $w$ and label each edge on the north and east boundary 0. The rest of the edge labels are uniquely determined by the edge local rules.

We associate a partition to the vertices of the growth diagram as follows.
\begin{enumerate}
	\item The north and east boundary vertices are $\emptyset$'s.
	\item If the edge $(i-1,j),(i,j)$ is labeled by $k$, then $\Gamma_w(i-1,j)$ is obtained from $\Gamma_w(i,j)$ by adding a box to the $k$-th column.
	\item  If the edge $(i,j-1),(i,j)$ is labeled by $k$, then $\Gamma_w(i,j-1)$ is obtained from $\Gamma_w(i,j)$ by adding a box to the $k$-th column.
\end{enumerate}
We will then read off the $P$-tableau from the west boundary, from top to bottom. The $Q$-tableau is obtained from the south boundary, reading from \emph{right} to \emph{left}. This defines a map $\RS':w\mapsto(P'(w), Q'(w))$.

\begin{figure}
	\resizebox{0.65\linewidth}{0.65\linewidth}{
	\input{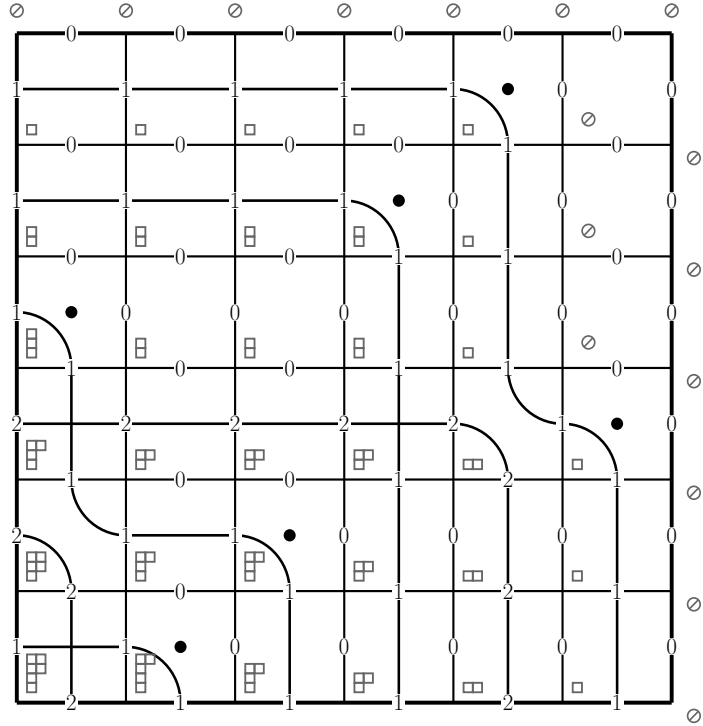}}
	\caption{Dual growth diagram}
	\label{fig:dual_growth}
\end{figure}
For $w=365214$, $\RS':w\mapsto (P'(w),Q'(w))$,
where $P'(w)=\ytableaushort{14, 25,3,6}$ and $Q'(w)=\ytableaushort{12,36,4,5}$ \footnote{We note that for this $w$ we have $Q(w)=Q'(w)$,  because this tableau is self-dual under the evacuation map.}.

\begin{prop}\label{prop:dual_growth}
	Let $\RS(w)=(P(w),Q(w))$ and $\RS'(w)=(P'(w),Q'(w))$ denote RS and dual RS map, respectively. Then we have $P(w)=P'(w)$ and $Q(w) = \textup{evac}(Q'(w))$, where $\textup{evac}$ is the evacuation operation as defined in \cite[Definition A1.2.8]{ec2}.
\end{prop}

\begin{proof}
	The edge labels of the dual growth diagram for $w$ are the horizontal flip of the classical growth diagram of $ww_0$, where $w_0$ is the longest element in $S_n$. Interpreting the edge labels via row versus column operations corresponds to transposing the diagram. The claim then follows from the fact that $\RS(ww_0)=(P^t, \operatorname{evac}(Q)^t)$, see \cite[Corollary 1.2.11]{ec2}.
\end{proof}

\section{Affine Growth Diagram}
\label{sec:affine}
Let $\tilde S_n$ be the extended affine Weyl group of type $A_{n-1}$. Namely,
\[\tilde{S}_n=\{w:\mathbb{Z}\to \mathbb{Z}\ \big|\ w\text{ is a bijection and }w(i+n)=w(i)+n\text{ for all }i\}.\]
For $w\in \tilde S_n$, we may specify $w$ by the window notation, $w=[w(1),\dots, w(n)]$.
Also let 
\[\tilde{S}^{(i)}_n\coloneqq\{w\in \tilde{S}_n:\sum_{j=1}^nw(j)-j=ni\}.\]
We define the growth diagram of $w$, denoted $\Gamma_w$, to be the $\mathbb{Z}\times\mathbb{Z}$ grid with the following additional structures. As before, we index the square given by the four vertices $(i-1,j-1),(i-1,j),(i,j-1),(i,j)$ to be the $(i,j)$-tile, and let the $(i,j)$-tile be marked if $w(j)=i$.
	
	If $u,v$ are two adjacent vertices in $\Gamma_w$, we use the notation $[u,v]$ to denote the edge between $u$ and $v$. Similarly, if $u$ and $v$ are vertices that lie on the same horizontal or vertical line, we use the notation $[u,v]$ to denote the set of edges between $u$ and $v$.

The construction of the growth diagram consists of two steps: we first use the Fomin-Viennot edge local rule to construct the edge-labeled growth diagram along with the uncolored shadow lines, then in the second step we introduce a rule to assign colors to the shadow lines, which will enable translation between growth diagrams and tableaux.

\subsection{Affine growth diagram and  shadow lines}
Each edge of the growth diagram $\Gamma_w$ will be labeled by a nonnegative integer as follows.
\begin{enumerate}
	\item An edge $[(i-1,j-1),(i-1,j)]$ or $[(i-1,j),(i,j)]$ is labeled 0 if there is no mark in any $(i',j')$-tile for all $i'\le i$, $j'\ge j$. 	
	\textcolor{blue}{}
	\item The other tiles are filled iteratively, using the same edge local rule as in Figure~\ref{fig:dual_edge_rule}.

	 \item For each tile whose edges are not all 0, connect the midpoints of edges with the same label:
	 
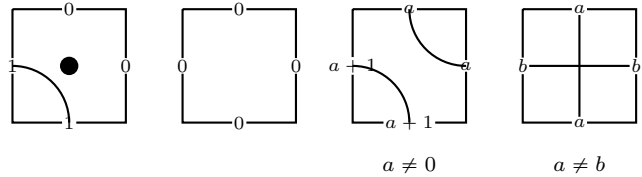
\[\begin{tikzpicture}[tikzfig]
	\begin{pgfonlayer}{nodelayer}
		\node [style=none] (0) at (-9, 3) {};
		\node [style=label-s] (1) at (-8, 3) {$0$};
		\node [style=none] (2) at (-7, 3) {};
		\node [style=label-s] (3) at (-7, 2) {$0$};
		\node [style=none] (4) at (-7, 1) {};
		\node [style=label-s] (5) at (-8, 1) {$1$};
		\node [style=none] (6) at (-9, 1) {};
		\node [style=label-s] (7) at (-9, 2) {$1$};

		\node [style=none] (8) at (-6, 3) {};
		\node [style=label-s] (9) at (-5, 3) {$0$};
		\node [style=none] (10) at (-4, 3) {};
		\node [style=label-s] (11) at (-4, 2) {$0$};
		\node [style=none] (12) at (-4, 1) {};
		\node [style=label-s] (13) at (-5, 1) {$0$};
		\node [style=none] (14) at (-6, 1) {};
		\node [style=label-s] (15) at (-6, 2) {$0$};

		\node [style=none] (16) at (-3, 3) {};
		\node [style=label-s] (17) at (-2, 3) {$a$};
		\node [style=none] (18) at (-1, 3) {};
		\node [style=label-s] (19) at (-1, 2) {$a$};
		\node [style=none] (20) at (-1, 1) {};
		\node [style=label-s] (21) at (-2, 1) {$a+1$};
		\node [style=none] (22) at (-3, 1) {};
		\node [style=label-s] (23) at (-3, 2) {$a+1$};

		\node [style=none] (24) at (0, 3) {};
		\node [style=label-s] (25) at (1, 3) {$a$};
		\node [style=none] (26) at (2, 3) {};
		\node [style=label-s] (27) at (2, 2) {$b$};
		\node [style=none] (28) at (2, 1) {};
		\node [style=label-s] (29) at (1, 1) {$a$};
		\node [style=none] (30) at (0, 1) {};
		\node [style=label-s] (31) at (0, 2) {$b$};

		\node [style=particle] (32) at (-8, 2) {};
		\node [style=label] (33) at (-2, 0.25) {$a\neq 0$};
		\node [style=label] (34) at (1, 0.25) {$a\neq b$};
	\end{pgfonlayer}

	\begin{pgfonlayer}{edgelayer}
		\draw (0.center) -- (1) -- (2.center) -- (3) -- (4.center) -- (5) -- (6.center) -- (7) -- (0.center);
		\draw (8.center) -- (9) -- (10.center) -- (11) -- (12.center) -- (13) -- (14.center) -- (15) -- (8.center);
		\draw (16.center) -- (17) -- (18.center) -- (19) -- (20.center) -- (21) -- (22.center) -- (23) -- (16.center);
		\draw (24.center) -- (25) -- (26.center) -- (27) -- (28.center) -- (29) -- (30.center) -- (31) -- (24.center);

		-> right
		\draw (7) arc (90:0:1);     


		\draw (17) arc (180:270:1); 
		\draw (23) arc (90:0:1);    

		\draw (25) -- (29);         
		\draw (31) -- (27);         
	\end{pgfonlayer}
\end{tikzpicture}
\]
\vskip -0.9em
\noindent We call the third displayed tile a \emph{bump tile} and the fourth a \emph{cross tile}. 
	 
\end{enumerate}
A \emph{shadow line} is an infinite line that runs in the northwest-southeast direction formed by connected straight lines or arcs in a filled affine growth  diagram. 
Given a shadow line $s$, let $\L(s)$ denote its label. If a shadow $s$ line passes through the $(i,j)$-tile, we say that the $(i,j)$-tile \emph{supports} $s$.
We will also associate a partition (Young diagram) to each vertex of the growth diagram as follows. Denote this partition $\Gamma_w(i,j)$.
\begin{enumerate}
	\item $\Gamma_{i,j}=\emptyset$ if there does not exist a marked tile indexed by $(a,b)$ with $a\leq i$ and $b\ge j$.
	\item If the edge $[(i,j),(i,j-1)]$ (resp. $[(i,j),(i-1,j)]$) is labeled by $k$, then $\Gamma(i,j)$ is the partition obtained by adding a box to $\Gamma(i,j-1)$ (resp. $\Gamma(i-1,j)$) at the $k$-th column.
\end{enumerate}

\begin{definition}
\label{def:col-label}
	Let $w\in\tilde S_n$, and denote $\lambda_{i,j}=\Gamma_w(i,j)$. 
	 For a fixed $j\in\mathbb{Z}$ and each $i\in\mathbb{Z}$, the sequence of partitions $\{\lambda_{t,j}\}_{t\leq i}$ determines a tableau $P_{i,j}=P_{i,j}(w)$ as follows: 
	 let $e_k$ be the $k$-th edge with a nonzero label in $[(-\infty,j),(i,j)]$ and suppose $\L(e_k)=a$, then
the tableau $P_{i,j}$ has the defining property that $\col_{P_{i,j}}(k)=a$.
\end{definition}

\begin{lemma}
\label{lem:affine_insertion}
	Let $(w_k)_{k>j,w_k\leq i}$ denote the finite word obtained by reading $w(j+1),w(j+2),\dots$, ignoring the entries with value greater than $i$. Then $P_{i,j}(w)$ agrees with the insertion tableau $P((w_k)_{k> j,w_k\leq i})$.
	For each $j$, the infinite tableau $\vec P_{j}(w):=\lim_{i\to\infty}P_{i,j}(w)$ is the same as the tableau obtained by inserting the semi-infinite word $(w_k)_{k>j}$. 
\end{lemma}
\begin{proof}
Given $(i,j)$, consider the subdiagram of $\Gamma_w$ consisting of vertices of index $(a,b)$ with $a\le i$ and $b\ge j$. Denote this diagram $\Gamma^{(i,j)}$. Note that it consists of marked boxes $(a,b)$ when $w(b)=a$ with $b> j$ and $a\leq i$. Since edge labels of the growth diagram have infinitely many $0$'s from the north and to the east, the semi-infinite growth diagram $\Gamma^{(i,j)}$ is equivalent to a finite dual growth diagram after removing the empty rows and columns. Since the dual affine growth diagram uses the same `dual growth rules', by \Cref{prop:dual_growth}, we have that $P_{i,j}(w)=P((w_k)_{k> j,w_k\leq i})$, where $P$ denotes the classical insertion tableau. Then 
\[\vec{P}_j=\lim_{i\to\infty}P_{i,j}=\lim_{i\to\infty}P((w_k)_{k> j,w_k\leq i})=P((w_k)_{k> j})\qedhere\]
\end{proof}

\begin{corollary}
	There are at most $n$ shadow lines labelled with 1.
\end{corollary}
\begin{proof}
By Definition~\ref{def:col-label} and Lemma~\ref{lem:affine_insertion}, for each $i,j$ the number of edges labelled 1 in $[(-\infty, j), (i,j)]$ is the length of the first column of $P_{i,j}(w)$, equivalently the length of a longest decreasing subsequence of $(w_k)_{k>j,w_k\leq i}$, which is at most $n$ since $w\in \tilde{S_n}$.

Since every shadow line extends infinitely in the northwest-southeast direction, it must cross the vertical line $[(-\infty,j),(\infty,j)]$ for every $j$. Thus the claim follows.
\end{proof}
\subsection{Colored Shadow Lines and Dual Affine RS}
We now introduce a coloring rule and use it to define the dual affine RS correspondence. 
Let $N$ be the number of shadow lines labeled with 1. We will define a coloring $\C(s)\in \{1,\dots,N\}$ for each shadow line $s$.

If $s$ and $s'$ are two non-crossing shadow lines and there exists a bump tile in which both $s$ and $s'$ pass through, we say that $s$ and $s'$ are \emph{repelling} shadow lines. Notice that if $\L(s')\le \L(s)$, then necessarily $\L(s)=\L(s')+1$. We also say in this case that $s'$ \emph{repels} $s$. \begin{lemma}
\label{lem:unique-repel}
	For every shadow line $s$ with $\L(s)>1$, there is a unique $s'$ that repels $s$.\end{lemma}

\begin{proof}
Recall that this property is apparent for any finite growth diagram by Viennot's  original shadow line construction (using the dual convention). The affine growth diagram is an asymptotic limit of finite growth diagrams (as $\lim_{\substack{i\rightarrow \infty, \\j\rightarrow -\infty}}\Gamma^{(i,j)}$ where $\Gamma^{(i,j)}$ is defined as in the proof of \Cref{lem:affine_insertion}), and thus the property still holds.

\end{proof}

\begin{definition}
\label{def:coloring-rule}	
We define a coloring rule on the shadow lines as follows:
\begin{enumerate}
	\item 
	Order the shadow lines $s^1_1,\dots, s^1_N$ labeled by 1 from northeast to southwest.		\item For each $1\le i\le N$, let $(s_i^j)_{j\in\mathbb{N}_+}$ be a sequence of shadow lines such that $s_i^j$ repels $s_i^{j+1}$  for all $j$. By Lemma~\ref{lem:unique-repel}, each shadow line is contained in a unique such sequence, and $\L(s_i^j)=j$ for all $j$. Assign each shadow line in the sequence the $i$-th color, i.e., 
		$\C(s^j_i)\coloneq i$ for all $j\in \mathbb{N}_+$ \footnote{The mnemonic of our notation is L for labeL and R for coloR. We adopt this choice to remind the reader that coloRs remember information about Rows of tableaux, and labeLs remember information about coLumns, as will be explained in the next section.}.
	\end{enumerate}
Notice that a shadow line is uniquely determined by its label and color. Given a color $r$ and label $l$, let $\S(r,l)\coloneqq s_r^l$ denote the corresponding shadow line. If $e$ is an edge crossed by a shadow line $s$, we define $\S(e)\coloneqq s$, $\L(e)\coloneqq\L(s)$ and $\C(e)\coloneqq \C(s)$.
\end{definition}
\begin{remark}
	Note that each shadow line carries two different data: a label and a color. There are finitely many colors, but infinitely many labels. \end{remark}

\begin{definition}
We define for $a,b\in\mathbb{Z}$ the $(a,b)$-\emph{window}, denoted by $\Omega^{(a,b)}$, to be the collection of tiles of the following form \[\{(i,j)\text{-tile}:(a-1)n< i \le  an,\ (b-1)n< j\leq bn \}.\] 
 We say that $\Omega^{(a,b)}$ is \emph{full} whenever $w(j)\le (a-1)n$ for all $(b-1)n<j\le bn$. Equivalently, a window is full if and only if the tiles within are cross tiles or bump tiles.
We define the $m$-th \emph{standard window} to be $\Omega^m\coloneqq \Omega^{(m,1)}$, and define $\lambda^m$ to be the partition at the northeast corner of $\Omega^m$, namely $\lambda^m=\Gamma_w((m-1)n,n)$.
\end{definition}
Observe that it is an immediate consequence of the periodicity $w(q+n)=w(q)+n$ for all $q\in \mathbb{Z}$ that in $\Gamma_w$, the $(i,j)$-tile  is identical to the $(i+n,j+n)$-tile for all $i,j\in \mathbb{Z}$. Therefore $\Omega_w^{(a,b)}$ is identical to $\Omega_w^{(a+1,b+1)}$ for all $a,b\in \mathbb{Z}$. It follows that the entire growth diagram $\Gamma_w$ is determined by the standard windows $\{\Omega_i\}_{i\in \mathbb{Z}}$.

\begin{definition}
	Let $\Omega=\Omega^{m}$ be a standard window with no marked tile below. Let $\oto$ be the sequence of edges of the north boundary, indexed from \emph{right} to \emph{left} (Important!). In other words, for $1\le i\le n $, $\oto(i)=[((m-1)n,n-i),((m-1)n,n-i+1)]$. Meanwhile, let $\ori$ be the sequence of edges of the east boundary, indexed from top to bottom, i.e. $\ori(i)=[((m-1)n+i-1, n),((m-1)n+i , n)]$.
\end{definition}

\begin{definition}
	 Given a sequence of edges $E=(E(1),\dots, E(m))$, we may define its \emph{label sequence} $\L(E)\coloneqq(\L(E(1)),\dots,\L(E(m)))$ and similarly its \emph{color sequence} $\C(E)$.
\end{definition}

\begin{definition}
	We say that a window is \emph{stable} if  the color sequence of the east boundary is the same as that of the west boundary (in other words, $\C(\Omega^i_\textup{east})=\C(\Omega^{i+1}_\textup{east})$), and if $c<d$ both appear in a color sequence, $c$ appears at least as many times as $d$. (Notice that the condition on color sequences of the boundary is equivalent to the condition that two shadow lines supported by a bump tile must have the same color. Furthermore, by \Cref{lem:stable-partition}, stability is equivalent to the condition $\lambda^{i+1}-\lambda^i=\lambda^{i+2}-\lambda^{i+1}$ and $\lambda^{i+1}-\lambda^i$ is a partition.)
\end{definition}

Notice that a stable window must also have $\C(\Omega^i_\textup{north})=\C(\Omega^{i+1}_\textup{north})$, since the condition on the vertical boundaries implies that no bump tile can support two shadow lines of different colors, so the color sequences on the horizontal boundaries must also agree.

We are now ready to define the dual affine RS correspondence. Recall that a \emph{tabloid} is a filling of a partition of size $n$ with numbers in $[n]$ such that each number appears exactly once, and strictly increasing in each row.

\begin{definition}\label{def:dars}
	Let $w\in \tilde S_n$ and $\Gamma_w$ be its growth diagram, and $\Omega^{N_0}$ be the first stable window of $\Gamma_w$. The dual affine RS correspondence is the map $\RS(w)=(\bar P,\bar Q,\lambda, N_0)$ obtained as follows:
	\begin{enumerate}
		\item $\bar P$ is a tabloid whose $i$-th row contains numbers $\{j:\C(\ori(j))=i\}$
		\item $\bar Q$ is a tabloid whose $i$-th row contains numbers $\{j:\C(\oto(j))=i\}$
		\item $\lambda = \lambda^{N_0}$, the partition at the northeast corner of $\Omega^{N_0}$.
		
	\end{enumerate}
\end{definition}

To prepare for the statement of the main theorem we give a few more definitions.
\begin{definition}
	Let $\bar{P}$ be a tabloid of shape $\mu$, and $\lambda$ be a partition. We define $\lambda \looparrowright \bar{P}$ to be the row-increasing filling of the skew shape $(\lambda+\mu)/\lambda$ such that for each $i$, the $i$-th row of $(\lambda+\mu)/\lambda$ and the $i$-th row of $\bar{P}$ have the same content. 
\end{definition}

\begin{example}
Let $\lambda = (2,1)$,
	\ytableausetup
{boxsize=normal,tabloids}
$\bar{P} =\ytableaushort{
245, 13, 6
},$ then \ytableausetup
{boxsize=normal,notabloids} $\lambda \looparrowright \bar{P}=\ytableaushort{\none\none245,\none13,6}$.
\end{example}

If $P$ is a (skew) standard tableau with $n$ boxes, for $1\le i\le n$, let $\row_P(i)$ denote the row index of the box containing $i$ in $P$, and $\col_P(i)$ the column index of the box containing $i$ in $P$.
 	\begin{definition}
 	\label{def:omega-saturated}
 		Let $P$ and $Q$ be skew standard tableaux of shape $\nu/\lambda$ with $|\nu/\lambda|=n$. 
 		We construct an $n\times n$ grid $\omega=\omega(P,Q)$ with north-east running colored shadow lines as follows.
 		\begin{enumerate}
 			\item Label and color the north and east edges of $\omega$ such that $\L(\omega_{\text{north}}(i))=\col_Q(i)$, $\C(\omega_{\text{north}}(i))=\row_Q(i)$,  $\L(\omega_{\text{east}}(i))=\col_P(i)$, and $\C(\omega_{\text{east}}(i))=\row_P(i)$. 
 			\item Determine the edge labels by the growth rule (Figure~\ref{fig:dual_edge_rule}).
 			\item Form north--east running shadow lines by connecting the same edge labels. The colors of these shadow lines are determined by the colors on the boundary.
 		\end{enumerate}
	We say that  $\omega$ is \emph{saturated} if $\nu-\lambda$ is a partition, and no square containing a bump has two colors.
 		\end{definition}
\begin{example}
\label{ex:omegaPQ}
	In \Cref{fig:omegaPQ}, the diagram on the left is $\omega(P,Q)$ for \ytableausetup{boxsize=1.2em}$P=\ytableaushort{\ \ 14,\ 23}$ and $Q=\ytableaushort{\ \ 23,\ 14}$ and this  is not saturated. The diagram on the right is $\omega(P,Q)$ for \[P=\ytableaushort{\ \ \ \ 134,\ \ \ 2} \text{\ \ \  and\ \  }Q=\ytableaushort{\ \ \ \ 123,\ \ \ 4}\] and this is saturated.
\end{example}
 	
\begin{figure} 	
\begin{tikzpicture}[x=1cm,y=1cm,line cap=round,line join=round]
  \tikzset{elabel/.style={font=\scriptsize, fill=white, inner sep=1pt}}
\definecolor{linecolor1}{HTML}{0072B2}
\definecolor{linecolor2}{HTML}{D55E00}
\definecolor{linecolor3}{HTML}{009E73}
\definecolor{linecolor4}{HTML}{CC79A7}

  \draw[gray!65] (0,0) grid (4,4);
  \draw[thick] (0,0) rectangle (4,4);

  \draw[line width=1.15pt,linecolor1]
    (4,3.5) -- (3,3.5)
    arc[start angle=270,end angle=180,radius=.5];

  \draw[line width=1.15pt,linecolor2]
    (4,2.5)
    arc[start angle=270,end angle=180,radius=.5]
    -- (3.5,4);

  \draw[line width=1.15pt,linecolor2]
    (4,1.5)
    arc[start angle=270,end angle=180,radius=.5]
    arc[start angle=0,end angle=90,radius=.5]
    -- (1,2.5)
    arc[start angle=270,end angle=180,radius=.5]
    -- (0.5,4);

  \draw[line width=1.15pt,linecolor1]
    (4,0.5)
    arc[start angle=270,end angle=180,radius=.5]
    arc[start angle=0,end angle=90,radius=.5]
    arc[start angle=270,end angle=180,radius=.5]
    -- (2.5,3)
    arc[start angle=0,end angle=90,radius=.5]
    arc[start angle=270,end angle=180,radius=.5];

  \foreach \x/\y/\lab in {
    0.5/4/3, 1.5/4/4, 2.5/4/3, 3.5/4/2,
    0.5/3/3, 1.5/3/5, 2.5/3/4, 3.5/3/2,
    0.5/2/4, 1.5/2/5, 2.5/2/4, 3.5/2/3,
    0.5/1/4, 1.5/1/6, 2.5/1/5, 3.5/1/4,
    0.5/0/4, 1.5/0/7, 2.5/0/6, 3.5/0/5}
    \node[elabel] at (\x,\y) {$\lab$};

  \foreach \x/\y/\lab in {
    0/3.5/5, 1/3.5/5, 2/3.5/4, 3/3.5/3, 4/3.5/3,
    0/2.5/4, 1/2.5/3, 2/2.5/3, 3/2.5/3, 4/2.5/2,
    0/1.5/6, 1/1.5/6, 2/1.5/5, 3/1.5/4, 4/1.5/3,
    0/0.5/7, 1/0.5/7, 2/0.5/6, 3/0.5/5, 4/0.5/4}
    \node[elabel] at (\x,\y) {$\lab$};
\end{tikzpicture} \qquad \qquad
\begin{tikzpicture}[x=1cm,y=1cm,line cap=round,line join=round]
  \tikzset{elabel/.style={font=\scriptsize, fill=white, inner sep=1pt}}
\definecolor{linecolor1}{HTML}{0072B2}
\definecolor{linecolor2}{HTML}{D55E00}
\definecolor{linecolor3}{HTML}{009E73}
\definecolor{linecolor4}{HTML}{CC79A7}

  \draw[gray!65] (0,0) grid (4,4);
  \draw[thick] (0,0) rectangle (4,4);

  \draw[line width=1.15pt,linecolor1]
    (4,3.5) 
    arc[start angle=270,end angle=180,radius=.5];

  \draw[line width=1.15pt,linecolor2]
    (4,2.5)
    -- (1,2.5)arc[start angle=270,end angle=180,radius=.5]
    -- (0.5,4);

  \draw[line width=1.15pt,linecolor1]
    (4,1.5)
    arc[start angle=270,end angle=180,radius=.5]
    -- (3.5,3)
    arc[start angle=0,end angle=90,radius=.5]
    arc[start angle=270,end angle=180,radius=.5];

  \draw[line width=1.15pt,linecolor1]
    (4,0.5)
    arc[start angle=270,end angle=180,radius=.5]
    arc[start angle=0,end angle=90,radius=.5]
    arc[start angle=270,end angle=180,radius=.5]
    -- (2.5,3)
    arc[start angle=0,end angle=90,radius=.5]
    arc[start angle=270,end angle=180,radius=.5];

  \foreach \x/\y/\lab in {
    0.5/4/4, 1.5/4/7, 2.5/4/6, 3.5/4/5,
    0.5/3/4, 1.5/3/8, 2.5/3/7, 3.5/3/6,
    0.5/2/5, 1.5/2/8, 2.5/2/7, 3.5/2/6,
    0.5/1/5, 1.5/1/9, 2.5/1/8, 3.5/1/7,
    0.5/0/5, 1.5/0/10, 2.5/0/9, 3.5/0/8}
    \node[elabel] at (\x,\y) {$\lab$};

  \foreach \x/\y/\lab in {
    0/3.5/8, 1/3.5/8, 2/3.5/7, 3/3.5/6, 4/3.5/5,
    0/2.5/5, 1/2.5/4, 2/2.5/4, 3/2.5/4, 4/2.5/4,
    0/1.5/9, 1/1.5/9, 2/1.5/8, 3/1.5/7, 4/1.5/6,
    0/0.5/10, 1/0.5/10, 2/0.5/9, 3/0.5/8, 4/0.5/7}
    \node[elabel] at (\x,\y) {$\lab$};
\end{tikzpicture}
\caption{Examples of $\omega(P,Q)$}
\label{fig:omegaPQ}
 	\end{figure}
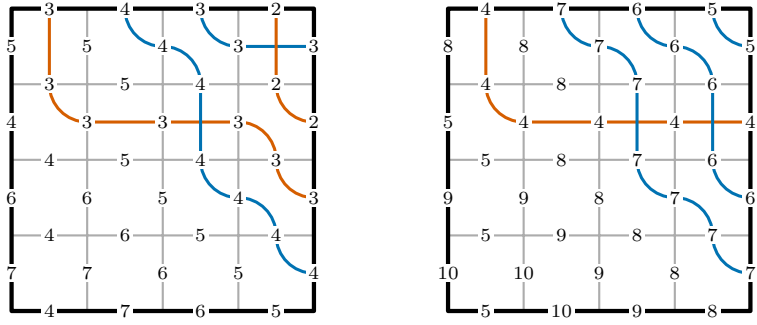
\begin{theorem}\label{thm:main}
	The map $\RS$ in \Cref{def:dars} is a bijection between $\tilde S_n$ and tuples $(\bar P,\bar Q,\lambda, N_0)$ with the following properties:
	\begin{enumerate}
		\item $\bar{P},\bar{Q}$ are tabloids of the same shape $\mu \vdash n$.
		\item $\lambda$ is a partition such that 
		
			\begin{enumerate}
				\item $\ell(\lambda)\leq \ell(\mu)$, where $\ell(\lambda)$ is the number of parts of $\lambda$.
	
				\item  $P\coloneqq \lambda\looparrowright \bar{P}$ and $Q\coloneqq \lambda\looparrowright \bar{Q}$ are both skew SYTs, and it is not possible to slide any column $i$ of $P$ and $Q$ simultaneously up by one square and maintain that both tableaux are skew SYTs.
				\item $\omega(P,Q)$ is saturated.
			\end{enumerate}
		\item If $\lambda':=\lambda-\mu$ is a partition, it does not satisfy all the conditions in (2).
		
		\item $N_0\in\mathbb{Z}$. 
	\end{enumerate}
	Furthermore, if $\RS(w)=(\bar P,\bar Q,\lambda, N_0)$ and $w\in \widetilde{S}^{(i)}_n$, then $|\lambda|=n(N_0-2)-i$.
\end{theorem}

 \begin{example}
\begin{figure}

	\input{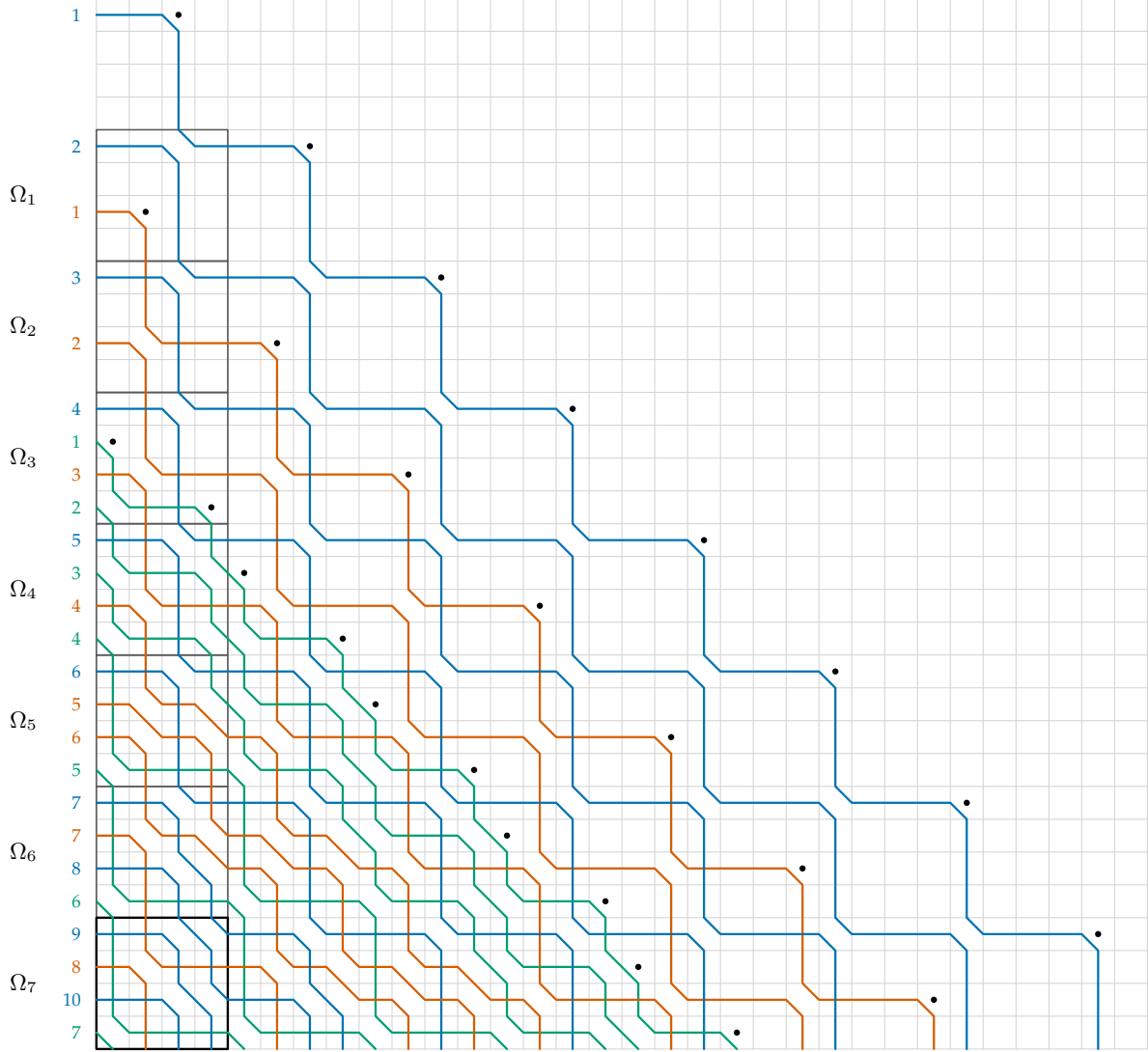}

\caption{Dual affine growth diagram for $w=[10,3,-3,12]\in S_4^{(3)}$. The labels for the shadow lines are shown on the left.}
\label{fig:mainex}
\end{figure}
In \Cref{fig:mainex}, the dual affine growth diagram computes the map $\RS(w)=(\bar{P},\bar{Q},\lambda,N_0)$ where $w=[10,3,-3,12]\in S_4^{(3)}$\ytableausetup
{boxsize=normal,tabloids}
$\bar{P} =\ytableaushort{
13,2,4
}$, $\bar{Q} =\ytableaushort{
12,3,4
}$, $\lambda = (6,6,5)$, and $N_0=7. $
\end{example}
\begin{example}
	Here we give some examples and non-examples of $(\bar P,\bar Q,\lambda)$ with respect to the conditions in \Cref{thm:main}. If \[\bar{P}=\ytableaushort{2,1},\; \bar{Q}=\ytableaushort{2,1},\; \lambda=(1),\] then condition (2)(b) is violated. If instead $\lambda=(2)$ then all conditions are satisfied. If instead $\lambda=(3,1)$ then condition (3) is violated. 
	If \[\bar{P}=\ytableaushort{14,23},\; \bar{Q}=\ytableaushort{23,14},\; \lambda=(2,1),\] then by \Cref{ex:omegaPQ}, condition (2)(c) is violated. On the other hand,  \[\bar{P}=\ytableaushort{134,2},\; \bar{Q}=\ytableaushort{123,4},\;\lambda=(4,3)\] satisfy all conditions.
\end{example}
\begin{remark}
	We emphasize that the coloring rule is a salient feature of the dual affine growth diagram. In particular, when $w$ is a finite permutation treated as an affine permutation under the natural embedding $S_n\hookrightarrow \widetilde{S}_n^{(0)}$, the partition $\lambda$ is always empty and the first stable window is always $\Omega^2$. The coloring rule only makes sense if we treat this finite permutation as an affine permutation!
\end{remark}
 		Given any pair $P,Q$ of skew SYTs and an integer $m$, it is possible to construct a growth diagram by labelling the edges on the north and east boundaries of $\Omega^m$ such that $\L(\Omega_{\mathrm{north}}(i))=\col_Q(i)$ and $\L(\Omega_{\mathrm{east}}(i))=\col_P(i)$, and then the rest of the edge labels can be uniquely determined by the  growth rules in Figure~\ref{fig:dual_edge_rule}. However, a priori, the coloring of the shadow lines as defined in \Cref{def:coloring-rule} is not guaranteed to agree with the coloring in the definition of $\omega(P,Q)$. (The precise condition needed to ensure the consistency of the coloring is given in condition (2)(b) in \Cref{thm:main}.)
 
We define the reverse map, $\RS^\dagger$.
\begin{definition}
\label{def:reverse}
	Let $(\bar{P}, \bar{Q},\lambda, N_0)$ be a tuple that satisfies the conditions in \Cref{thm:main}. We construct a growth diagram as follows:
	\begin{enumerate}
	\item Let $P=\lambda \looparrowright \bar{P}$ and $Q=\lambda \looparrowright \bar{Q}$.
		\item Place edge labels such that  $\L(\Omega_{\mathrm{north}}^{N_0}(i))=\col_Q(i)$ and $\L(\Omega^{N_0}_{\mathrm{east}}(i))=\col_P(i)$.
		\item Fill the $N_0$-th standard window $\Omega^{N_0}$ using the edge local rules in \Cref{fig:dual_edge_rule}. Fill the rest of the standard windows $\Omega^m=\Omega^{(m,1)}$ with the growth rule by the condition that $\Omega^m_\mathrm{west}$ and $\Omega^{m+1}_\mathrm{east}$ have the same sequence of labels. (Notice that for each standard window, if the north and east boundaries are given, the rest of the edge labels in this window are uniquely determined by the growth rule, and the same is true if the west and south boundaries are given.)
		\item Determine the rest of the growth diagram by $\Omega^{(a,b)}=\Omega^{(a+1, b+1)}$ for all $a,b\in\mathbb{Z}$. 
		\item The tiles marked with a dot determine the affine permutation $w$. 
	\end{enumerate}
	Let $\RS^\dagger(\bar{P}, \bar{Q},\lambda, N_0)\coloneqq w$.
\end{definition}

The following proposition shows that the growth diagram for an affine permutation is determined by a small amount of local information.

\begin{prop}
\label{prop:reverse-from-label}
	Let $\mathbf{a}, \mathbf{a}'$ be sequences of positive integers of length $n$ with the same contents (namely, same as multisets). Then we may construct a growth diagram  by placing edge labels on $\Omega^{N_0}$ such that $\L(\Omega^{N_0}_\mathrm{north}(i))=\mathbf{a}_i$ and $\L(\Omega^{N_0}_\mathrm{east}(i))=\mathbf{a}'_i$ and filling the diagram by following the procedure in \Cref{def:reverse}(3)(4). Then the tiles marked with a dot determine uniquely an affine permutation $w$.
\end{prop}

\begin{proof}
	Connect the edge labels in each tile as before. The growth rule ensures that in $\Omega^{N_0}$, there cannot be a horizontal line across a row of tiles, or a vertical line across a column of tiles: suppose to the contrary that there is a horizontal line $l$ connecting $\Omega^{N_0}_\mathrm{east}(i)$ to $\Omega^{N_0}_\mathrm{west}(i)$ and the edge label connected by this line is $a$. Since the north boundary and the west boundary have the same multiset of edge labels, it must be the case that an edge with label $a$ on the north boundary is connected with an edge with the same label on the south boundary, which must cross the horizontal line $l$, which violates the growth rule because shadow lines with the same label cannot cross. This implies that all edges  on the west boundary are connected to the edges on the south boundary. By the construction of the growth diagram, the same property holds for every standard window. 
	
	We show that every column contains exactly one tile marked with a dot. Consider the tiles in column $i$  in $\Omega^m$. If $m> N_0$, all edge labels on these tiles are positive. If $m\le N_0$, by the observation in the preceding paragraph, if $\L(\Omega_\mathrm{south}^m(i))$ is nonzero, there must be a bump tile or a marked tile in column $i$ of $\Omega^m$. Looking north in each column, a bump tile decreases horizontal edge labels by 1. Therefore, there must be a unique window $\Omega^{m_0}$, $m_0<N_0$ that contains a marked tile, and all tiles above this tile have all edge labels 0. Therefore there is a unique marked tile in column $i$ for $1\le i \le n$. By periodicity of the growth diagram this is true for every column. The argument for every row contains exactly one tile marked with a dot is similar.\end{proof}

\section{Proof of Main Results}
\label{sec:proof}

\begin{lemma}
	\label{lem:color=row}
	Fix $j\in \mathbb{Z}$ and let $(e_i)_{i\in \mathbb{Z}_+}$ be the sequence of edges  that $e_i$ is the $i$-th edge with a nonzero label in $[(-\infty,j), (\infty,j)]$.
	 If there exists a shadow line that crosses $e_i$, then $\row_{\vec P_{j}}(i)=\C(e_i) $. In other words, the number $i$ is in the $\C(e_i)$-th row of the infinite tableau $\vec P_j$.
\end{lemma}

\begin{proof}
	 Recall that $\S(e_i)$ is the shadow line that goes through $e_i$.
	We first prove that
	the label on the edge $e_i$ is the $\C(e_i)$-th time that $\L(e_i)$ appears in the sequence $\{\L(e_i)\}_{i\in\zz}$. 
	
	For the first color, it suffices to show that $\C(e_i)=1$ if and only if $e_i$ is the first time on which $\L(e_i)$ appears. In other words, all edges on which a label first appears must be colored by $1$. Necessarily $\L(e_1)=1$. By \Cref{lem:unique-repel}, there exists a unique subsequence of edges $e_1=f_1,f_2,\cdots$ such that $\S(f_{i})$ repels $\S(f_{i+1})$ for each $i$. By construction these are all colored $1$. We will show that $f_i$ must be the first edge with label $i$. Suppose not and let $k$ be the smallest index such that $f_k$ is  not the first edge with label $k$, and let $g_k$  be the first edge with label $k$. Then there is a unique sequence of repelling shadow lines $\S(g_1),\cdots,\S(g_k)$ with $\L(g_i)=i$. Since $f_{k-1}$ is the first edge with label $k-1$, we must have that $g_{k-1}$ comes after $f_{k-1}$. So these edges are in the order of $f_{k-1},g_{k-1},g_k,f_k$. Because $\S(f_{k-1})$  is repelled by $\S(f_{k})$, we must have that $\S(g_{k-1})$ and $\S(g_k)$  cross either $\S(f_{k-1})$ or $\S(f_k)$. Since shadow lines of the same label cannot cross, we have that $\S(f_{k-1})$ crosses $\S(g_k)$ and $\S(f_k)$ crosses $\S(g_{k-1})$. But this will imply that $\S(g_{k-1})$  crosses with $\S(g_k)$, contradicting  the assumption that $\S(g_{k-1})$ is repelled by $\S(g_k)$.
	
	We then assume by induction that this claim is true for all colors up to $N-1$. For the $N$-th color, the proof will follow a similar logic as the case of the first color. Now let $f_1$ be the $N$-th edge with label 1; it is by definition colored by $N$. Then the  sequence $f_1,f_2,\cdots$ such that $\S(f_i)$ repels $\S(f_{i+1})$ are all colored by $N$. It suffices to show that $f_i$ is the $N$-th occurrence of an edge labelled by $i$. By induction hypothesis, each $f_i$ cannot be the $t$-th edge labelled $i$ for any $t<N$; thus it's left to show that each $f_i$ cannot be the $t$-th edge with label $i$ with $t>N$. Suppose otherwise, and suppose $k$ is the smallest number such that $f_k$ is the $M$-th edge with label $k$ for $M>N$. Let $g_k$ be the $N$-th edge with label $k$, and suppose $\S(g_{k-1})$  repels $\S(g_k)$. Then the edges $f_{k-1},g_{k-1},g_k,f_k$ will be in the exact same contradictory situation as in the first color case.
This completes the proof of the claim.

Recall that $\L(e_i)$ is the column index of $i$ in $\vec P_j$, so if the edge $e_i$ is the $\C(e_i)$-th edge with label $\L(e_i)$, then it must be the case that $i$ is the $\C(e_i)$-th entry of the $\L(e_i)$-th column. Thus $i$ must be on the $\C(e_i)$-th row, completing the proof. 	\end{proof}

\begin{corollary}
\label{cor:at-most-n-rows}
	For all $(i,j)\in\mathbb{Z}\times \mathbb{Z}$, the partition $\Gamma_w(i,j)$ has at most $n$ rows.
\end{corollary}

\begin{proof}

	It directly follows from Lemma~\ref{lem:color=row} that the number of rows in each partition is at most the number of colors, which is at most $n$.
\end{proof}

\begin{prop}\label{prop:crossing}
	Every cross tile has the property that the shadow line with smaller color has larger label.
\end{prop}
\begin{proof}
Let the $(i,j)$-tile be a cross tile, and $\lambda\coloneqq\Gamma_w(i-1,j)$ be the partition at the northeast corner of this tile. Suppose the shadow lines $s$ and $t$ cross in this tile and $\C(s)<\C(t)$. Let $A$ be the set of shadow lines that cross $[(i-1,j), (\infty, j)]$, $A_s\coloneqq \{x\in A: \C(x)=\C(s)\}$, and $A_t\coloneqq \{x\in A: \C(x)=\C(t)\}$. Observe that $s$ is the shadow line with the smallest label $a\coloneqq \L(s)$ in $A_s$ and $t$ is the shadow line with the smallest label $b\coloneqq \L(t)$ in $A_t$. By Definition~\ref{def:col-label} and Lemma~\ref{lem:color=row}, it follows that $\lambda_{\C(s)}=\L(s)-1$ and $\lambda_{\C(t)}=\L(t)-1$. Since $\lambda$ is a partition and $\L(s)\neq \L(t)$, we must have $\L(s)>\L(t)$.
\end{proof}

Let $\mu, \mu'$ be compositions of $n$. We say that $\mu$ \emph{dominates} $\mu'$ if for all $1\le k\le n$, $\mu_1+\dots+\mu_k\ge \mu_1'+\dots+\mu_k'$.

Recall that  $\lambda^i$ is the partition at the northeast corner of the $i$-th standard window. Let $\mu^i=\lambda^{i+1}-\lambda^{i}$. Let $N$ be the index of the first standard window such that all dots of the permutation are strictly above this window, namely $w(j)\le n(N-1)$ for $1\le j\le n$ .
\begin{lemma}
\label{lem:stable-partition}
	 For any $i\ge N$, we have $\mu^{i+1}$ dominates $\mu^i$. The window
	 $\Omega^i$ is stable if and only if $\mu^i=\mu^{i+1}$ and $\mu^i$ is a partition.
	 Furthermore, there exists some $m\ge N$ such that $\mu^m$ is a partition and  $\mu^i=\mu^m$ for all $i\ge m$.
	 \end{lemma}
\begin{proof}
	By assumption we have that $\mu^i$ is a composition of $n$ for all $i\ge N$. Consider the $\rho$-th row in the $i$-th standard window with $i\ge N$ for some $1\le \rho\le n$, and let $r$ be the color of the rightmost edge in this row and $r'$ be the color of the leftmost edge in this row. Let $s$ and $s'$ be two shadow lines supported by a bump tile on this row, and assume $\L(s')=\L(s)+1$. If $s$ and $s'$ are repelling, they have the same color; otherwise $s$ and $s'$ must cross somewhere, in which case by Proposition~\ref{prop:crossing}, it must be the case that $\C(s')<\C(s)$. It follows that if $r\neq r'$, then necessarily $r'<r$. The east boundary of the $(i+1)$-th standard window is identical to the west boundary of the $i$-th standard window. Now imagine adding boxes to the partition at the northeast corner of a standard window as  we go down the east boundary of this window. If in the $i$-th standard window, at the $\rho$-th step a box is added to some row $r$, by the previous argument, in the $(i+1)$-th standard window, at the $\rho$-th step a box must be added to a row $r'\le r$. The first claim then follows.
	
	For the second statement, the ``only if'' direction is clear. Now if $\mu^i=\mu^{i+1}$, by the  argument from the previous paragraph we must have that every bump tile in $\Omega^i$ supports two shadow lines of the same color, and thus the color sequence of the east boundary is the same as the color sequence of the west boundary. If in addition $\mu^i$ is a partition, a larger color cannot appear more times than a smaller color in the color sequence. It follows that $\Omega^i$ is a stable window. 
	
	Since the number of compositions of $n$ is finite, there must exist some $m\ge N$ such that  $\mu^i=\mu^m$ for all $i\ge m$. If $\mu^m$ is not a partition, it is then impossible that $\lambda^i$ is a partition for all $i\ge m$, a contradiction. 
	\end{proof}

\begin{lemma}
\label{lem:stability}
	If $\Omega^i$ is a stable window, then $\Omega^{i+1}$ is also a stable window.
	\begin{proof}
		Let $e$ be an edge in $\Omega^i$ and denote by $e'$  the corresponding edge in $\Omega^{i+1}$ (that is, $e'$ has end vertices with the same column indices and the row indices increase by $n$). We claim that if $e$ has label $l$ and color $r$, then $e'$ has label $l+\mu^i_r$ and color $r$. 
		 We check that this labeling of the edges in $\Omega^{i+1}$ indeed satisfies the growth rule. Suppose $e$ and $f$ are perpendicular edges of the same tile in $\Omega^i$ and $\L(e)<\L(f)$ (namely, the tile containing $e$ and $f$ is a cross tile), we must have $\C(e)>\C(f)$ and since $\mu^i$ is a partition, we have $\mu^i_{\C(e)}\le \mu^i_{\C(f)}$. Thus $\L(e)+\mu^i_{\C(e)}<\L(f)+\mu^i_{\C(f)}$ and the tile containing $e'$ and $f'$ is a cross tile. If $e, f, g, h$ are top, right, bottom  and left edges of the same tile in $\Omega^i$ and $\L(e)=\L(f)$ (namely, the tile containing $e$ and $f$ is a bump tile), we must have $\L(g)=\L(h)=\L(e)+1$, and by the definition of stable window, $\C(e)=\C(f)=\C(g)=\C(h)$. Therefore,  $\L(e)+\mu^i_{\C(e)}=\L(f)+\mu^i_{\C(f)}$ and $\L(g)+\mu^i_{\C(g)}=\L(h)+\mu^i_{\C(h)}=\L(e)+\mu^i_{\C(e)}+1$, meaning that the tile containing $e', f', g', h'$ is a bump tile.
		 
		 It remains to show that the color sequence of the west boundary of $\Omega^{i+1}$ is the same as that of the west boundary of $\Omega^i$. Consider an edge $e$ colored with 1 in the west boundary of $\Omega^i$ and suppose its label is $l$. Then this is the first appearance of $l$ among the edges in $[(-\infty, 0), (\infty,0)]$. We claim that the edge $e'$ with label $l+\mu^i_1$ is also the first appearance of the label $l+\mu^i_1$. Because edge labels correspond to column indices of the insertion tableau whose shape is a partition, the label $l+\mu_1^i$ cannot appear on the west boundary of $\Omega^i$. Furthermore, if $f$ is an edge before $e$, then necessarily $\L(f)<l$, and thus $\L(f)+\mu^i_{\C(f)}<l+\mu^i_1$. It follows that $e'$ is colored 1.
		 	 To consider the second color, we may disregard all edges colored by the first color (so that the second color becomes the new first color) and repeat the argument. The rest of the colors follow the same way. 
		 	 \end{proof}
\end{lemma}

Let $w\in \tilde{S}_n$ and $\Omega=\Omega^m$ be the $m$-th standard window in $\Gamma_w$ and suppose that it is full. We define $P(\Omega)$ and $Q(\Omega)$ to be the skew SYTs of shape $\lambda^{m+1}/\lambda^m$ such that $\col_{P(\Omega)}(i)=\Omega_{\mathrm{north}}(i)$ and $\col_{Q(\Omega)}(i)=\Omega_{\mathrm{east}}(i)$ for all $1\le i\le n$. In what follows (up to the end of \Cref{prop:lambda-minimal}), we characterize the image of the map $\Omega\mapsto (P(\Omega), Q(\Omega))$ in the set of all pairs of skew SYTs.

Let $\mathbf{a}$ be a sequence consisting of $a$ and $a+1$. Let $d(\mathbf{a})$ be the smallest nonnegative integer such that for all $i>0$, if the $(d(\mathbf{a})+i)$-th occurrence of $a+1$ exists, there must be at least $i$ occurrences of $a$ before this $a+1$. If $\mathbf{a}$ and $\mathbf{a}'$ have the same content, namely, they are two sequences consisting of $a$ and $a+1$ that are identical as multisets, let $\delta(\mathbf{a},\mathbf{a}')=\max(d(\mathbf{a}),d(\mathbf{a}'))$.

It is an immediate consequence of the definition of $d(\mathbf{a})$ that if 
$P$ is a skew SYT such that $\col_P(i)=\mathbf{a}_i$
for all $i$, the minimum number of blank boxes in column $a$ is $d(\mathbf{a})$. Similarly, if we also require that $Q$ is a skew SYT of the same shape as $P$ and $\col_Q(i)=\mathbf{a}'_i$, then the minimum number of blank boxes in column $a$ of $P$ and $Q$ is $\delta(\mathbf{a},\mathbf{a}')$.
\begin{example}
	Suppose $\mathbf{a}_1=(2,3,2,3,3,2,3,2,3,2)$, $\mathbf{a}_2=(3,3,2,2,2,2,2,3,3,3)$. Then $d(\mathbf{a}_1)=1$, $d(\mathbf{a}_2)=2$, and $\delta(\mathbf{a}_1,\mathbf{a}_2)=2.$
\end{example}

\begin{lemma}
\label{lem:smallest-lambda}
	Let $T$ be a sequence of positive integers. Suppose $b_i$ is the $i$-th largest number that occurs in $T$. Let $\mathbf{b}_i$ be the subsequence of $T$ consisting of $b_{i}-1$ and $b_i$. Let $\lambda$ be the partition such that the $k$-th column of (the Young diagram of) $\lambda$ contains 
	$d(\mathbf{b}_1)+\dots +d(\mathbf{b}_i)$ boxes, where $i$ is the largest index such that $b_i\ge k+1$. Let $P$ be a skew SYT of shape $\mu/\nu$ such that $\col_P(i)=T_i$. Then $\lambda$ is the smallest such $\nu$.  
	\end{lemma}

\begin{proof}

Initialize an empty grid with $m$ columns, where $m$ is the maximum number that occurs in $T$.
Suppose $T$ has length $l$.
	Reading $T$ from left to right, for each $1\le i \le l$, we put the number $i$ in the first empty spot in column $T_i$. Clearly the numbers in each column are in increasing order from top to bottom. Now, starting from the  last nonempty column and moving to the left, we add a minimal number of blank boxes (possibly zero) to the top of each column and push the numbers downwards to ensure row strictness towards the right, and after we are done we will get a skew SYT with the desired property. We now show that the number of boxes we should add to each column is exactly as stated.
	
	Consider the $k$-th column, for $k=m,\dots, 1$. When $k=m$, no $i$ satisfies $b_i-1\ge m$, since $m$ is the largest number in $T$, so the sum is 0. Indeed we should add no box to the $m$-th column. Suppose the claim is true for $k$ and we show it for $k-1$. To ensure row strictness towards the right we must first add the same number of boxes as in the previous step, which is $d(\mathbf{b}_1)+\dots +d(\mathbf{b}_i)$ boxes, where $i$ is the largest index such that $b_i\ge k+1$. By the definition of $k$ we must have $b_{i+1}\le k$. If $b_{i+1}< k$, $k$ does not appear in $T$, so $P$ should only have empty boxes in column $k$, and this agrees with the rule.  If $b_{i+1}=k$, it must be that $b_{i+2}<k$, so $i+1$ is the largest index such that $b_{i+1}\ge (k-1)+1=k$. By the observation after the definition of $d$, we must add $d(\mathbf{b}_{i+1})$ more blank boxes to column $k$, as predicted by the rule.
\end{proof}
\begin{example}
	Let $T=(4,3,3,5,1,8,3,5,6,6)$. Then $(b_1,b_2,b_3,b_4,b_5,b_6)=(8,6,5,4,3,1)$, $\mathbf{b}_1=(8)$, $\mathbf{b}_2=(5,5,6,6)$, $\mathbf{b}_3=(4,5,5)$, $\mathbf{b}_4=(4,3,3,3)$, $\mathbf{b}_5=(3,3,3)$, $\mathbf{b}_6=(1)$, $d(\mathbf{b}_1)=1$, 
	$d(\mathbf{b}_2)=0$, $d(\mathbf{b}_3)=1$, $d(\mathbf{b}_4)=1$, $d(\mathbf{b}_5)=3$, $d(\mathbf{b}_6)=1$. Then $\lambda^t_7=d(\mathbf{b}_1)=1$, $\lambda^t_6=d(\mathbf{b}_1)=1$, $\lambda^t_5=d(\mathbf{b}_1)+d(\mathbf{b}_2)=1$, $\lambda^t_4=d(\mathbf{b}_1)+d(\mathbf{b}_2)+d(\mathbf{b}_3)=2$, $\lambda^t_3=d(\mathbf{b}_1)+d(\mathbf{b}_2)+d(\mathbf{b}_3)+d(\mathbf{b}_4)=3$, $\lambda^t_2=d(\mathbf{b}_1)+d(\mathbf{b}_2)+d(\mathbf{b}_3)+d(\mathbf{b}_4)+d(\mathbf{b}_5)=6$, $\lambda^t_1=d(\mathbf{b}_1)+d(\mathbf{b}_2)+d(\mathbf{b}_3)+d(\mathbf{b}_4)+d(\mathbf{b}_5)=6$. The corresponding skew tableau is shown below.
	\begin{center}
	\ytableausetup{boxsize=1.2em, notabloids}
	\begin{ytableau}
{} & {} & {} & {} & {} & {} & {} & 6  \\
{} & {} & {} & {} &4 & 9 \\
{} & {} & {} & 1     &8 & 10 \\
{} & {} & 2 \\
{} & {} & 3 \\
{} & {} & 7 \\
5
\end{ytableau}
\end{center}
\end{example}

We have the analogous version of Lemma~\ref{lem:smallest-lambda} for a pair of sequences. The proof is identical, with $d$ replaced by $\delta$, so we omit the proof.
\begin{lemma}
\label{lem:smallest-lambda-for-pair-of-sequences}
	Let $T,T'$ be sequences of positive integers with the same content. Suppose $b_i$ is the $i$-th largest number that occurs in $T$ and $T'$. Let $\mathbf{b}_i$ (resp. $\mathbf{b}_i'$) be the subsequence of $T$ (resp. $T'$) consisting of $b_{i}-1$ and $b_i$. Let $\lambda$ be the partition such that the $k$-th column of (the Young diagram of) $\lambda$ contains 
	$\delta(\mathbf{b}_1,\mathbf{b}_1')+\dots +\delta(\mathbf{b}_i,\mathbf{b}_i')$ boxes, where $i$ is the largest index such that $b_i\ge k+1$. Let $P,Q$ be  skew SYTs of shape $\mu/\nu$ such that $\col_P(i)=T_i$ and $\col_Q(i)=T_i'$. Then $\lambda(T,T')\coloneqq \lambda$ is the smallest such $\nu$.
\end{lemma}

 \begin{lemma}
 \label{lem:non-intersect}
	Suppose $a$ and $a+1$ both appear in $\L(\Omega_{\mathrm{east}})$. Let $r_1<\dots <r_k$ be indices of $a$ or $a+1$ in $\L(\Omega_{\mathrm{east}})$, and let $t_1<\dots<t_k$ be indices of $a$ or $a+1$ in $\L(\Omega_{\north})$. Let $\mathbf{a}_{\mathrm{east}}=(\L(\Omega_{\mathrm{east}}(r_1)),\dots, \L(\Omega_{\mathrm{east}}(r_k)))$, and let $\mathbf{a}_{\north}=(\L(\Omega_{\north}(t_1)),\dots, \L(\Omega_{\north}(t_k)))$. Suppose the total number of $a+1$ in $\mathbf{a}_\mathrm{east}$ is $q$. Let $\delta=\delta(\mathbf{a}_\mathrm{east}, \mathbf{a}_\north)$, and suppose $\delta<\delta+i\le q$. Then the shadow line $s_{i+\delta}$ that passes through the $(i+\delta)$-th $a+1$ in $\mathbf{a}_\mathrm{east}$ (and also in $\mathbf{a}_\north$) does not intersect the shadow line $s'_{i}$ that passes through the $i$-th $a$ in $\mathbf{a}_\mathrm{east}$ (and also in $\mathbf{a}_\north$).
\end{lemma}
\begin{proof}
Let $\Omega'$ be the standard window above $\Omega$. 
We  show  that for each $i$, $s_{i+\delta}$ and $s'_i$ do not intersect in $\Omega$, and separately $s_{i+\delta}$ and $s'_i$ do not intersect in $\Omega'$. These two facts imply that $s_{i+\delta}$ and $s_i'$ do not intersect.

	We first show that $s_{\delta+1}$ and $s'_1$ do not intersect in $\Omega$. Suppose $s'_1$ connects $\Omega_\mathrm{east}(y_1)$ and $\Omega_\north (x_1)$, and $s_{\delta+1}$ connects $\Omega_\mathrm{east}(y_2)$ and $\Omega_\north (x_2)$. By assumption, $y_1$ is the position of the first $a$ in $\Omega_\mathrm{east}$, $y_2$ is the position of the $(\delta+1)$-th $a+1$ in $\Omega_\mathrm{east}$, $x_1$ is the position of the first $a$ in $\Omega_\north$, and $x_2$ is the position of the $(\delta+1)$-th $a+1$ in $\Omega_\north$. In particular, $s_1'$ is the northeast-most shadow line in $\Omega$ labelled $a$. By definition of $\delta$, we must have $y_1<y_2$ and $x_1<x_2$. 
	If they did intersect, then $s_{\delta+1}$ must contain a southwest elbow that is strictly to the northeast of $s'_1$. This means that there must be another shadow line labelled $a$ that is strictly northeast of $s_1'$ in $\Omega$, which is a contradiction.
	
	Now suppose by induction that  $s_{\delta+i}$ and $s_i'$ appear but do not intersect in $\Omega$. Suppose $s'_i$ connects $\Omega_\mathrm{east}(y_1)$ and $\Omega_\north (x_1)$, and $s_{\delta+i}$ connects $\Omega_\mathrm{east}(y_2)$ and $\Omega_\north (x_2)$. By the same reasoning as before we have $y_1<y_2$ and $x_1<x_2$. Suppose $s_{\delta+i+1}$ connects $\Omega_\mathrm{east}(y_3)$ and $\Omega_\north(x_3)$. Then we must have $y_2<y_3$ and $x_2<x_3$ and $s_{\delta+i+1}$ is strictly southwest of $s_{\delta+i}$, since they have the same label. Now suppose $s'_{i+1}$ connects $\Omega_\mathrm{east}(y')$ and $\Omega_\north(x')$. Notice that $x_1<x'$ and $y_1<y'$,  $s'_{i+1}$ is strictly southwest of $s'_i$, and there are no other shadow lines labelled $a$ in between $s'_i$ and $s'_{i+1}$. Furthermore, $x'<x_3$ and $y'<y_3$ by the definition of $\delta$. If $s_{\delta+i+1}$ and $s'_{i+1}$ were to intersect, then $s_{\delta+i+1}$ must contain a southwest elbow that is strictly northeast of $s'_{i+1}$, and in the same tile $T$ the northeast elbow is labeled $a$. But this elbow cannot be part of $s'_{i}$, because $T$ is strictly southwest of $s_{\delta+i}$, and since $s_{\delta+i}$ and $s'_i$ do not intersect, $s'_i$ is strictly northeast of $s_{\delta+i}$. This means that there is a shadow line labeled $a$ strictly between $s'_{i}$ and $s'_{i+1}$, which is a contradiction.
	
	To show the non-intersection in $\Omega'$, 
	we first show $s_q$ does not intersect  $s_{q-\delta}'$. Suppose $s_{q}$ connects $\Omega'_
	\mathrm{west}(y_2)$ and $\Omega'_\mathrm{south}(x_2)$ and $s'_{q-\delta}$ connects $\Omega'_
	\mathrm{west}(y_1)$ and $\Omega'_\mathrm{south}(x_1)$ for some $x_1,x_2,y_1,y_2$. By assumption, $y_2$ is the position of the last ($q$-th) $a+1$ in $\Omega'_{\mathrm{south}}=\Omega_{\mathrm{\north}}$, $y_1$ is the position of the $(q-\delta)$-th $a$ in $\Omega'_{\mathrm{south}}$, and therefore $y_1<y_2$. Similarly, $x_1<x_2$. If $s_q$ intersects with $s_{q-\delta}'$, there must be a bump tile southwest of $s_{q-\delta}$ supporting $s_{q-\delta}'$ and another shadow line labeled $a+1$, which contradicts that $s_q$ is the southwest-most shadow line in $\Omega'$.
	
	Now suppose by induction that $s_{\delta+i+1}$ and $s'_{i+1}$ do not intersect in $\Omega'$ for some $i\ge 1$ with $\delta+i+1\le q$. Suppose $s_{\delta+i+1}$ connects $\Omega'_
	\mathrm{west}(y_2)$ and $\Omega'_\mathrm{south}(x_2)$ and $s'_{i+1}$ connects $\Omega'_
	\mathrm{west}(y_1)$ and $\Omega'_\mathrm{south}(x_1)$ for some $x_1,x_2,y_1,y_2$. By similar reasoning as before we have $x_1<x_2$ and $y_1<y_2$. Now suppose $s'_i$ connects $\Omega'_\mathrm{west}(y_0)$ and $\Omega'_\mathrm{south}(x_0)$. Then necessarily $y_0<y_1$ and $x_0<x_1$ and $s'_i$ is strictly northeast of $s'_{i+1}$. Now suppose $s_{\delta+i}$ connects $\Omega'_
	\mathrm{west}(y')$ and $\Omega'_\mathrm{south}(x')$. By the assumption on $\delta$, we must have $y_0<y'<y_2$ and $x_0<x'<x_2$. If $s_{\delta+i}$ intersects with $s'_i$, then $s_i'$ must go through a bump tile $T$ as a northeast elbow, and $T$ must be strictly northeast of $s'_{i+1}$. Since $s_{\delta+i+1}$ is strictly southwest of $s_{i+1}'$, $T$ must support another pipe with label $a+1$ that is different from $s_{\delta+i+1}$ that is strictly between $s_{\delta+i+1}$ and $s_{\delta+i}$, which is a contradiction.
	\end{proof}

\begin{prop}
\label{prop:lambda-minimal}
Suppose $\Omega^m$ is a full standard window of $\Gamma_w$ and $P,Q$ are a pair of skew SYTs of shape $\mu/\lambda$ such 
that $\col_P(i)=\Omega^m_{\mathrm{north}}(i)$ and $\col_Q(i)=\Omega^m_{\mathrm{east}}(i)$.
The partition $\lambda^m$ is the smallest possible partition among such $\lambda$. (In other words, by Lemma~\ref{lem:smallest-lambda-for-pair-of-sequences}, $\lambda^m=\lambda(\Omega^m_{\textup{north}},\Omega^m_{\textup{east}})$.)
	\end{prop}

	\begin{proof}
		The edge labels of the growth diagram by definition have the interpretation of adding a box at a certain column; therefore $\Omega_{\text{north}}$ and $\Omega_{\text{east}}$ can be interpreted as skew tableaux of shape $\lambda^{m+1}/\lambda^m$. Call these two skew tableaux $P,Q$, then by construction they have the property $\col_P(i)=\Omega_{\mathrm{north}}(i)$ and $\col_Q(i)=\Omega_{\mathrm{east}}(i)$. We are left to prove the minimality of $\lambda^m$.
		
		Suppose for contradiction that $\lambda'=\lambda(\Omega_\mathrm{north},\Omega_{\mathrm{east}})$ is strictly smaller than $\lambda^m$. Let $a$ be the largest column where $(\lambda')^t_a<(\lambda^m)^t_a$. 
		Let $\delta_a=\delta(\mathbf{a}_\mathrm{north},\mathbf{a}_\mathrm{east})$, where $\mathbf{a}_\mathrm{north}$
		is the subsequence of $\L(\Omega_\mathrm{north})$ consisting of $a$ and $a+1$, and $\mathbf{a}_\mathrm{east}$
		is the subsequence of $\L(\Omega_\mathrm{east})$ consisting of $a$ and $a+1$.
		By construction of $\lambda'$, we must have $\delta_a=(\lambda')^t_{a}-(\lambda')^t_{a+1}$. 
		Let $\rho_a=(\lambda^m)^t_{a}-(\lambda^m)^t_{a+1}$, then $\delta_a<\rho_a.$ By Lemma~\ref{lem:non-intersect}, the shadow line that passes through the $(1+\delta_a)$-th $a+1$ does not intersect the shadow line that passes through the first $a$ and is strictly southwest. By the coloring rule and Lemma~\ref{lem:color=row} the shadow line that passes through the first $a$ should repel the shadow line that passes  through the $(1+\rho_a)$-th $a+1$, a contradiction.
 	\end{proof}

 \begin{corollary}
 \label{cor:color-consistent}
 	If $(\bar{P},\bar{Q},\lambda, N_0)$ satisfies that
 	$\bar P,\bar Q$ are tabloids of the same shape $\mu\vdash n$, 
		 $\ell(\lambda)\leq \ell(\mu)$, $N_0\in \mathbb{Z}$, $P=\lambda\looparrowright \bar P$ and $ Q=\lambda \looparrowright\bar Q$ are both skew SYTs, and it is not possible to slide any column $i$ of $P$ and $Q$ simultaneously up by one square and maintain that both tableaux are skew SYTs (namely, it satisfies the conditions (1), (2)(a), (2)(b), and (4) in \Cref{thm:main}), then the reverse map as defined in \Cref{def:reverse} produces a growth diagram whose colored shadow lines in the $N_0$-th standard window  agree with those in the definition of $\omega(P,Q)$.
 \end{corollary}

 	\begin{lemma}\label{lem:jing}
 		Let $\Omega=\Omega^m$ be a full, standard window in $\Gamma_w$. If $\Omega$ is a stable window, then $P(\Omega)=\lambda\looparrowright\bar{P}$, 
 		$Q(\Omega)=\lambda\looparrowright\bar{Q}$
 		for some tabloids $\bar{P}$, $\bar{Q}$, and partition $\lambda$ that satisfy conditions (1) and  (2) of Theorem~\ref{thm:main}. 
 		\end{lemma}
 	\begin{proof}
 	
 		Let $P=P(\Omega)$ and $Q=Q(\Omega)$ be the corresponding skew SYTs. If $\Omega$ is stable,  we must have $\lambda^{m+1}-\lambda^m$ is a partition. Let $\lambda =\lambda^m$, $\bar{P}$ and $\bar{Q}$ be tabloids of shape $\lambda^{m+1}-\lambda^m$ with the same row contents as $P$ and $Q$, respectively. By definition, these satisfy conditions (1) and (2)(a).
 		By Proposition~\ref{prop:lambda-minimal}, it is not possible to slide any column $i$ of $P$ and $Q$ simultaneously by one box without violating the conditions of skew SYTs, and thus (2)(b) is satisfied.    Notice that $\omega(P,Q)$ by construction has the same edge labels and north-east running shadow lines. The stability condition ensures that no bump tile in $\Omega$ can have two distinct colors, and thus $\omega(P,Q)$ is saturated, satisfying (2)(c).   
 		\end{proof}

 \begin{lemma}
 \label{lem:omega-saturated}
 	Suppose  $(\bar{P},\bar{Q},\lambda, N_0)$ satisfies the conditions in \Cref{cor:color-consistent}, and furthermore $\omega(P,Q)$ is saturated (\Cref{def:omega-saturated}). Then $\Omega^{N_0}$ in the image of the reverse map is a stable window in $\Gamma_w$ for $w=\mathrm{RS}^\dagger(\bar{P},\bar{Q},\lambda, N_0)$. If it also satisfies (3) then $\Omega^{N_0}$ is the first stable window.
 \end{lemma}
 \begin{proof}
 	Recall that $\omega(P,Q)$ being saturated means that if two north-east running shadow lines in $\omega(P,Q)$ are supported by a same bump tile, then they must have the same color as defined in the construction of $\omega(P,Q)$. By \Cref{cor:color-consistent}, the colors in $\omega(P,Q)$ are consistent with the colors of the shadow lines in the definition of the growth diagram, it suffices to show that no bump tile in $\Omega^{N_0}$ supports two differently colored shadow lines. Since the north-east running shadow lines already satisfy this property, any potential violation must involve the west-south  running shadow lines. First notice that the shadow lines colored 1 cannot bump with a shadow line with a different color. If there were a west-south running shadow line $s$ not colored 1 that bumps with a shadow line $t$ colored with 1, it must be that $\L(s)>\L(t)$. If $t$ does not cross with $s$, $s$ would have been colored with 1, so it must be that $t$ crosses with $s$. Then by \Cref{prop:crossing}, it must be that $\C(s)<\C(t)$, which is not possible because 1 is the smallest color. Now, there must be $\mu_1$  shadow lines colored with 1 that travel from north to east, and $\mu_1$  shadow lines colored with 1 that travel from west to south. The last of such must have label $\lambda_1+2\mu_1$. All bumps involving any of these $2\mu_1$ shadow lines colored with 1 are the same color.  There are at most $2\mu_2\le 2\mu_1$ shadow lines of color 2, and thus the last of such has label at most $\lambda_2+2\mu_2\le \lambda_1+2\mu_1$. By similar reasoning as before, if there is a shadow line not colored 2 that bumps with a shadow line colored with 2, it must have a smaller color. To be colored 1, the label must be larger than $\lambda_1+2\mu_1$, so this is not possible. Further colors follow from a similar argument.
 	
 	For the last claim, if $\Omega^{N_0}$ is not the first stable window, by \Cref{lem:stability}, $\Omega^{N_0-1}$ is also a stable window. 
 	By the definition of stable windows, all stable windows give rise to the same tabloids $\bar{P}$ and $\bar{Q}$. 
 	By \Cref{lem:jing}, this contradicts that $(\bar{P},\bar{Q},\lambda, N_0)$ satisfies condition (3).
 \end{proof}
 
 \begin{corollary}
 	\label{cor:first-stable}
 	If $\Omega=\Omega^m$ is the first stable window in $\Gamma_w$, then $P(\Omega)=\lambda\looparrowright\bar{P}$, 
 		$Q(\Omega)=\lambda\looparrowright\bar{Q}$
 		for some tabloids $\bar{P}$, $\bar{Q}$, and partition $\lambda$ that satisfy condition (3) of Theorem~\ref{thm:main}
 \end{corollary}
 \begin{proof}
 	If condition (3) is not satisfied, then $\lambda^{m}-\lambda^{m-1}$ is a partition and satisfies all conditions in (2). By \Cref{lem:omega-saturated}, $\Omega^{m-1}$ is a stable window, contradicting that $\Omega^m$ is the first.
 \end{proof}
  The following two lemmas relate the index of the (extended) affine permutation $w$ with the size of $\lambda^m$, the partition at the northeast corner of a full standard window $\Omega^m$.

 \begin{lemma}
 \label{lem:eq}
 	For $w\in \tilde S_n^{(i)}$ and any $(a,b)$ with $a-b=i$, we have that%
 	\[\#\{t\;\big|\;t> b, w(t)\leq  a\} = \# \{t\;\big|\;t\leq b,w(t)> a\}.\]
 \end{lemma}
 
 \begin{proof}
  Define
$
\Delta_b(w)
=
\#\{t\le b\mid w(t)>a\}
-
\#\{t>b\mid w(t)\le a\}.
$
These sets are finite because $w(t)-t$ is $n$-periodic, hence bounded.
We prove that $\Delta_b(w)=0$ by induction on the length of $w$. Let
$
\tau_i(t)=t+i.
$
This is the unique length-zero element in $\widetilde S_n^{(i)}$. Every element of $\widetilde S_n^{(i)}$ can be written as
$
w=\tau_i s_{r_1}\cdots s_{r_\ell},
$
where the $s_{r_i}$'s are the affine simple reflections, and $\ell=\ell(w)$.

First consider the base case $w=\tau_i$. Then
$
w(t)=t+i.
$
Since $a=b+i$, we have
\[
w(t)>a
\iff
t+i>b+i
\iff
t>b.
\]
Thus there is no $t\le b$ with $w(t)>a$. Similarly, there is no $t>b$ with $w(t)\le a$. Hence
$
\Delta_b(\tau_i)=0.
$

It remains to show that $\Delta_b$ is unchanged by right multiplication by a simple reflection. Let $s_r$ be an affine simple reflection.  Then $ws_r$ is obtained from $w$ by swapping the values $w(p)$ and $w(p+1)$ for all $p\equiv r \pmod n$.
Indeed,  if $p+1\le b$ or $p>b$, then swapping $w(p)$ and $w(p+1)$ clearly does not change their total contribution to $\Delta_b$.
The only nontrivial case is when  $p=b$. Then the old contribution of the pair $\{b,b+1\}$ is
$
\mathbf 1_{w(b)>a} - \mathbf 1_{w(b+1)\le a}.
$
After multiplying by $s_r$, the values $w(b)$ and $w(b+1)$ are swapped, so the new contribution is
$ \mathbf 1_{w(b+1)>a}
- \mathbf 1_{w(b)\le a}.
$
But these two quantities are equal, since
$\mathbf 1_{x>a}=1-\mathbf 1_{x\le a}.$
Hence $\Delta_b(ws_r)=\Delta_b(w)$.
\end{proof}

 \begin{lemma}
 \label{lem:affine_perm_signature}
 	Suppose $w\in \tilde S_n^{(i)}$. For any full standard window $\Omega^m$ and $\lambda^m$ the partition at its northeast corner, we have $|\lambda^m| = n(m-2)-i.$
 \end{lemma}
 \begin{proof}
 	By Lemma~\ref{lem:affine_insertion}, we know that 
 	$\lambda^m$ is the shape of the tableau $P_{n(m-1),n}(w)=P((w_{k})_{k>n, w_k\le n(m-1)})$. Thus $|\lambda|=\#\{j\; \big|\;j>n, w(j)\le n(m-1)\}$.
 	
 	Now%
 	\begin{align*}
 		& \#\{j\; \big|\;j>n, w(j)\le n(m-1)\} \\
 		=& \#\{j\; \big|\;j>n, w(j)\le n+i\} + \#\{j\; \big|\;j>n, n+i< w(j)\le n(m-1)\} \\
 		=& \#\{j\; \big|\;j\le n, w(j)> n+i\} + \#\{j\; \big|\;j>n, n+i< w(j)\le n(m-1)\} \\
 		=&\#\{j\; \big|\;j\le n, n+i< w(j)\le n(m-1)\}+\#\{j\; \big|\;j>n, n+i< w(j)\le n(m-1)\} \\
 		=&  \#\{j\; \big|\; n+i< w(j)\le n(m-1)\} \\
 		=&n(m-2)-i,
 	\end{align*}
 	where the second equality is by Lemma~\ref{lem:eq}, and the third equality follows from the fact that $\Omega^m$ is full, i.e., all marked tiles in columns 1 through $n$ are strictly above this window.
 \end{proof}

\begin{proof}[Proof of \Cref{thm:main}]
	
	For the forward map, \Cref{lem:stable-partition}  shows that the first stable window exists, so $\RS$ is defined. \Cref{lem:jing} and \Cref{cor:first-stable} shows that conditions (1)-(3) in \Cref{thm:main} are satisfied and (4) is automatic. For the reverse map, \Cref{prop:reverse-from-label} ensures that the construction gives an affine permutation, and \Cref{lem:omega-saturated} ensures that the reverse map and forward map are compatible. Finally, \Cref{lem:affine_perm_signature} establishes the relation between the index of the permutation and the size of $\lambda$. 
	\end{proof}

\section{Kazhdan-Lusztig cells}

\label{sec:KL}

In this section, we show that the dual affine Robinson--Schensted correspondence can be used to parametrize the Kazhdan--Lusztig cells in affine type $A$, by comparing it with Shi's affine Robinson--Schensted correspondence. We begin with a brief review of the fundamentals of Kazhdan--Lusztig theory.

Let $(W,S)$ be a Coxeter system, where $W$ is a Coxeter group and $S$ is the set of simple reflections. Denote by $\ell : W \to \mathbb{Z}_{\ge 0}$ the length function. The Hecke algebra $\mathcal{H}$ is the $A = \mathbb{Z}[q^{1/2}, q^{-1/2}]$-algebra with basis $\{T_w : w \in W\}$ and multiplication given by
\[
T_w T_{w'} = T_{ww'} \quad \text{if } \ell(ww') = \ell(w) + \ell(w'),
\]
\[
(T_s + 1)(T_s - q) = 0 \quad \text{for } s \in S.
\]

Each $T_w$ is invertible, with
\[
T_s^{-1} = q^{-1}(T_s + (1-q)) \quad \text{for } s \in S.
\]

There is a natural involution $a \mapsto \overline{a}$ on $A$ defined by $\overline{q^{1/2}} = q^{-1/2}$, which extends to an involution on $\mathcal{H}$ via
\[
\overline{\sum a_w T_w} = \sum \overline{a_w} \, T_{w^{-1}}^{-1}.
\]

For $w \in W$, define $\varepsilon_w = (-1)^{\ell(w)}$ and $q_w = q^{\ell(w)}$. Equip $W$ with the Bruhat order: $w \le w'$ if and only if some reduced expression of $w$ is a subword of a reduced expression of $w'$.

\begin{theorem}[Kazhdan--Lusztig, {\cite[Theorem 1.1]{KL79}}]
There exists an $A$-basis $\{C_w : w \in W\}$ of $\mathcal{H}$ such that $\overline{C_w} = C_w$ and
\[
C_w = \sum_{w' \le w} \varepsilon_{w'} \varepsilon_w \, q_w^{1/2} q_{w'}^{-1} \, P_{w',w} \, T_{w'},
\]
where $P_{w,w} = 1$ and, for $y < w$, $P_{y,w} \in A$ is a polynomial in $q$ of degree at most $\frac{1}{2}(\ell(w) - \ell(y) - 1)$.
\end{theorem}

The basis $\{C_w : w \in W\}$ is called the \emph{Kazhdan--Lusztig basis} of $\mathcal{H}$. Using this basis, we define preorders $\le_L$, $\le_R$, and $\le_{LR}$ on $W$ generated as follows:
\begin{enumerate}
    \item $w \le_L w'$ if $C_w$ appears with nonzero coefficient in $C_s C_{w'}$ for some $s \in S$;
    \item $w \le_R w'$ if $C_w$ appears with nonzero coefficient in $C_{w'} C_s$ for some $s \in S$;
    \item $\le_{LR}$ is the preorder generated by $\le_L$ and $\le_R$.
\end{enumerate}

The preorder $\le_L$ induces an equivalence relation $\sim_L$ on $W$, where $w \sim_L w'$ if $w \le_L w'$ and $w' \le_L w$. The equivalence classes are called \emph{left cells}. Similarly, $\sim_R$ and $\sim_{LR}$ define \emph{right cells} and \emph{two-sided cells}, respectively. Moreover, $\le_L$ (resp. $\le_R$, $\le_{LR}$) induces a partial order on the set of left (resp. right, two-sided) cells. These cells play an important role in the representation theory of $\mathcal{H}$.

In type $A$, the cells are elegantly described by the classical Robinson--Schensted correspondence.

\begin{theorem}[{\cite{KL79}}]
Let $W = S_n$, and let $\RS(w) = (P(w), Q(w))$ be the classical Robinson--Schensted correspondence. Then, for $u,w \in S_n$,
\[
u \sim_L w \iff Q(u) = Q(w), \quad
u \sim_R w \iff P(u) = P(w), \quad
u \sim_{LR} w \iff \sh(P(u)) = \sh(P(w)),
\]
where $\sh(P)$ denotes the shape of the tableau $P$.
\end{theorem}

Motivated by this result, Shi introduced the \emph{affine Robinson--Schensted algorithm} $w \mapsto (\bar P', \bar Q')$, which associates to an affine permutation a pair of tabloids of the same shape. This construction was later extended to a bijection $w \mapsto (\bar P', \bar Q', \rho)$ by Honeywill \cite{honeywill} and given a more visual interpretation via the \emph{affine matrix ball construction} {\cite{ambc}}. 

Further, Chmutov-Pylyavskyy-Yudovina \cite{ambc} showed that the tabloid $\bar P'$ can be interpreted asymptotically as a limit of the classical Robinson-Schensted map.

\begin{lemma}[\cite{ambc}]\label{lem:ambc-limit}
	Let $w\in\tilde S_n$, and $w^{(i)}$ be the permutation $(w(0),w(1),\cdots,w(i))$, and let $\vec{P}_i=P(w^{(i)})$ where $P$ is the classical Robinson-Schensted map. Then for large enough $i$, we have that $\row_{\vec{P}_i}(k)=\row_{\bar{P}'(w)}(\bar k)$, where $\bar k = k\mod n$.
\end{lemma}

Let $\RS'(w)= (\bar P',\bar Q',\rho)$ be the affine RS correspondence, and $\RS(w)=(\bar P,\bar Q,\lambda,k)$ be the dual affine RS correspondence.
\begin{lemma}
\label{lem:P-agree}
	We have that $\bar P'(w) = \bar P(w)$ for all $w\in \tilde S_n$. Hence for any $w,u\in \tilde{S}_n^{(0)}$, $w\sim_{R} u\iff \bar P(u)=\bar P(w)$. 
\end{lemma}
\begin{proof}
	We  note that the asymptotic description of $\bar{P}'$ in \Cref{lem:ambc-limit} agrees with the asymptotic description of $\bar{P}$ via the dual affine growth diagram by \Cref{lem:affine_insertion,lem:color=row}. It is known from \cite[Section 7]{shi} and \cite[Theorem 9.4]{ambc} that $w\sim_R u \iff \bar{P}'(u)=\bar{P}'(w)$, and thus also equivalent to $\bar{P}(u)=\bar{P}(w)$. 
\end{proof}

In \cite{chmutov2022affine}, the authors introduced an involution on tabloids called \emph{affine evacuation}.
\begin{definition}[{\cite[Theorem 3.1]{chmutov2022affine}}]\label{def:evac}
	Let $r:\tilde S_n\to\tilde S_n$ be the reflection automorphism of the $\tilde A_{n-1}$ Dynkin diagram that sends $s_i\to s_{n-i}$. Then there is an involution ${\evac}$ on tabloids such that if $\RS'(w)=(\bar P',\bar Q',\rho)$ then $\RS'(r(w))=({\evac}(\bar P'),{\evac}(\bar Q'),\tau)$ for some other $\tau$. In particular, for any $w\in\tilde S_n$, we have $\bar P'(r(w))=\evac(\bar P'(w))$.
\end{definition}

\begin{remark}[{\cite[Proposition 3.25]{chmutov2022affine}}]
	One can also obtain affine evacuation asymptotically. For a tabloid $\bar T$, let $w$ be any affine permutation such that $\bar P(w)=\bar T$, then one gets an infinite tableau $\vec T$ where $\vec T_i=P(w(0),\cdots,w(i))$. Then applying finite evacuation (which we denote also by $\evac$, by abuse of notation) on each of the $\vec T_i$, we get a new infinite tableau $\vec S$ such that $\vec S_i:=\evac(\vec T_i)$. The evacuation of $\bar T$ can be defined to be the tabloid $\bar S$ such that $\row_{\vec S}(k)=\row_{\bar S}(\bar k)$ for large enough $k$.
\end{remark}

\begin{prop}\label{prop:symmetry}
	For any $w\in \tilde{S}_n$, $\bar P(w)=\bar Q(r(w^{-1}))$.
\end{prop}
\begin{proof}
The reflection on a rock diagram across the antidiagonal is the map $(i,j)\mapsto (n+1-j,n+1-i)$. Thus the composition of anti-diagonal reflection and diagonal reflection is $(i,j)\mapsto (n+1-j,n+1-i)\mapsto (n+1-i,n+1-j)$. Note that this is exactly the Dynkin reflection $r$; thus
the rock diagram of $r(w^{-1})$ is the same as that of $w$ after reflection across the anti-diagonal, which exchanges the $\bar P$-tabloid and $\bar Q$-tabloid. \end{proof}
	
\begin{corollary}
\label{cor:evac}
	$\bar Q(w)=\evac(\bar Q'(w))$.

\end{corollary}

	\begin{proof}
		By \Cref{def:evac} and \Cref{prop:symmetry}, we have $\bar Q(w)=\bar P(r(w^{-1}))=\bar P'(r(w^{-1}))=\evac \bar P'(w^{-1})=\evac (\bar Q'(w))$.
	\end{proof}

\begin{corollary}
\label{cor:Q}
	For any $w,u\in \tilde{S}_n^{(0)}$, $w\sim_L u$ if and only if $\bar Q(w)=\bar Q(u)$.
\end{corollary}
\begin{proof}Since $r$ is a Dynkin diagram automorphism we have $w\sim_R u\iff r(w)\sim_R r(u)$. 
	Then we have $\bar Q(w)=\bar Q(u)\iff \bar P(r(w^{-1}))=\bar P(r(u^{-1}))\iff r(w^{-1})\sim_R r(u^{-1})\iff w^{-1}\sim_R u^{-1}\iff w \sim_L u$. \end{proof}
	The following theorem summarizes the main result in this section. 
	\begin{theorem}
	\label{thm:dual}
	For all $w\in \tilde{S}_n$, we have that $\bar{P}(w)=\bar{P}'(w)$ and $\bar{Q}(w)=\evac(\bar{Q}'(w))$. As a consequence, for all $w,u\in \tilde{S}_n^{(0)}$, $w\sim_R u$ if and only if $\bar{P}(w)=\bar{P}(u)$, and $w\sim_L u$ if and only if $\bar{Q}(w)=\bar{Q}(u)$.
\end{theorem}

\begin{proof}
	This is obtained by combining \Cref{lem:P-agree}, \Cref{cor:evac}, and \Cref{cor:Q}.
\end{proof}

\section{Conjectural Geometric Interpretation}

\label{sec:geometry}
Let $\mathcal{K}\coloneqq\cc((t^{-1}))$ denote the  field of formal Laurent series, and 
$\mathcal{O}\coloneqq\cc[[t^{-1}]]$ the ring of formal power series\footnote{Our choice of using $t^{-1}$ instead of $t$ as the generator allows more natural correspondence with the combinatorial construction.}.
For $f\in \K$ where $f(t) =\sum_{i\ge N}a_i t^{-i}$ with $a_i\in \cc$, and $f\neq 0$, we let
$\ord(f)$ be the smallest integer for which $a_i\neq 0$.
Let $\{e_1,\cdots, e_n\}$ denote the standard $\K$-basis of $\K^n$, and
for $c\in \mathbb{Z}$, define $e_{j+nc}\coloneqq t^c e_j$.
A \textbf{lattice} $L \subset \K^n$
is a free $\O$-submodule generated by $v_1,\ldots , v_n$ where
$\{v_1,\dots ,v_n\}$ is a $\K$-basis of $\K^n$. 
Consider the family of \textbf{standard $\O$-lattices}
\[E_j \coloneqq \O\tup{e_j,e_{j-1},\cdots, e_{j-n+1}} = \mathrm{fSpan}_{\cc }\tup{e_k}_{k\le j},\]
where $\mathrm{fSpan}$ denotes the formal span of possibly infinitely many basis vectors.

The (type A) \textbf{affine flag variety} $Fl_n$ is the space of all chains
of lattices $L= (\cdots \subset L_1\subset L_2\subset \cdots \subset L_n\subset \cdots )$
such that $L_{i+n} = tL_i$ 
and $\dim(L_i/L_{i-1}) = 1$
for all $i$. 
The \textbf{standard flag} is $E \coloneqq (\cdots \subset E_1\subset\cdots\subset E_n\subset \cdots )$,
whose stabilizer $\I$ is the subgroup of $GL_n(\mathcal{O})$ which is
upper-triangular modulo $t^{-1}$:
\[\I = \{b=(b_{ij})\in GL_n(\O): \ord(b_{ij})>0\  \forall i>j\}.\]
 Therefore $Fl_n \cong GL_n(\K)/\I$, and $\I$ is the \textbf{ Iwahori
 subgroup} of $GL_n(\K)$.

There is a canonical bijection
\[GL_n(\K)\backslash (Fl_n\times Fl_n)\cong I\backslash Gl_n(\K)/I\cong \tilde{S}_n.\]
In other words, for each $w\in \tilde{S}_n$, define the orbit $O_w=GL_n(\K)\cdot (I,wI)\subset Fl_n\times Fl_n$, then $Fl_n\times Fl_n=\coprod_{w\in \tilde{S}_n}O_w$. The \textbf{relative position} of $F$ and $F'$ for any two flags $F,F'\in Fl_n$ is the element $w=w(F,F')\in \tilde{S}_n$ such that $(F,F')\in O_w$. Explicitly, the flag $wI$ as an infinite chain of lattices is $(E_{w(j)})_{j\in \mathbb{Z}}$.

For any flag $F\in Fl_n$, let%
\[\mathfrak{n}_F\coloneqq \{\eta\in \mathfrak{gl}_n(\K):\eta(F_i)\subseteq F_{i-1}\text{ for all }i\in \mathbb{Z}\}\]%
the (pro-)nilpotent radical of the Iwahori Lie algebra stabilizing $F$. 
Let $(F,F')\in Fl_n\times Fl_n$. Notice that if $\eta\in \mathfrak{n}_F\cap \mathfrak{n}_{F'}$, then $\eta$ induces a well-defined nilpotent operator on the finite dimensional vector space $F_i/(F_i\cap F_j')$ for all $i,j\in \mathbb{Z}$. 

We conclude this paper with the following conjecture.

\begin{conj}

\label{conj:main-geometry}
	Suppose $(F,F')\in Fl_n\times Fl_n$, and let $V_{ij} = F_i/(F_i\cap F_j')$ and $w=w(F,F')$, the relative position of $F$ and $F'$. There is an open set $\mathcal{N}$ of $\mathfrak{n}_{F}\cap \mathfrak{n}_{F'}$ such that for any $\eta\in \mathcal{N}$, the Jordan type $\operatorname{JCF}(\eta|_{V_{ij}})$ is $\Gamma_w(i,j)$, the partition at position $(i,j)$ in the growth diagram of $w$. 
\end{conj}

\bibliographystyle{amsalpha}
\bibliography{ref.bib}

\end{document}